\documentclass[11pt,reqno]{amsart}
\usepackage{amscd,amsmath,amsopn,amssymb,amsthm,multicol}
\usepackage{tikz,tikz-cd,anysize,verbatim,ifthen,xargs,colortbl,float}
\usepackage{longtable,mathtools, cleveref, tcolorbox, comment, graphicx, adjustbox}
\usepackage{enumitem}
\usepackage[english]{babel}
\usepackage[utf8]{inputenc}
\everymath=\expandafter{\the\everymath\displaystyle}

\usetikzlibrary{decorations.markings,arrows,arrows.meta,bending,calc}
\tikzset{>={Stealth[scale=1.5, bend]}}
\usepackage{makecell} \setcellgapes{6pt}

\theoremstyle{plain}
\newtheorem{theorem}{Theorem}

\newtheorem{lemma}{Lemma}

\theoremstyle{definition}
\newtheorem{definition}{Definition}

\newtheorem{remark}{Remark}

\newcommand\com[1]{}
\newcommand\C{{\mathbb C}}
\newcommand\D{\mathcal{D}}

\newcommand\op[1]{\mathop{\rm #1}\nolimits}
\newcommand\p{\partial}

\newcommand\R{{\mathbb R}}

\makeatletter
    \def\@@and{}
\makeatother

\title[Global invariant Euler--Lagrange equations]{On globally invariant Euler--Lagrange equations for curves}

\author{Boris Kruglikov}
\address{Department of Mathematics and Statistics, UiT the Arctic University of Norway, Troms\o\ 9037, Norway.
\ E-mail: {\tt boris.kruglikov@uit.no}.}
\author{Eivind Schneider}
\address{Department of Mathematics and Statistics, UiT the Arctic University of Norway, Troms\o\ 9037, Norway.
\ E-mail: {\tt eivind.schneider@uit.no}.}
\author{Wijnand Steneker}
\address{Department of Mathematics and Statistics, UiT the Arctic University of Norway, Troms\o\ 9037, Norway.
\ E-mail: {\tt wijnand.s.steneker@uit.no}.}

\date{}

\begin{document}

 \begin{abstract}
Invariant Lagrangians yield invariant Euler--Lagrange equations, and it was discussed in the 
literature how to compute those using various local methods. The focus of this paper is on global 
algebraic differential invariants. In this case the computation can be modified in several aspects.
We will discuss relations with previous approaches and some foundational aspects.

The theory of invariant Euler--Lagrange equations was applied to curves with respect to
the motion group in the Euclidean plane and space. We expand those computations 
to the next dimension four (Minkowski spacetime), which already exhibits computational challenges. 
We also provide formulas for other examples, namely the projective and conformal (M\"obius) groups 
and relate to some recent applications.
 \end{abstract}

\maketitle

\section{Introduction}

The calculus of variations aims at description of extremals of an action, usually written as
the least action principle (sometimes maximum principle, but these are equivalent):
 $$
\int L\,dt \to \min
 $$
where we omit any mention of the boundary conditions, one of the standard choices
is the case of fixed ends or fast decay at infinity. Here $L=L(t,u,u',\dots,u^{(n)})$ is
a Lagrangian (more precise term: Lagrangian density) depending on $u=u(t)$ and its derivatives up to
order $n$; in what follows we will use the jet notations $u_0=u$, $u_1=u'$, \dots, $u_n=u^{(n)}$.
 
We will only consider the problems with one-dimensional time $t$ (leaving aside multiple integrals)
but several dependent variables $u=(u^i)$, in other words, we will discuss higher order variational
problems for curves on manifolds. 
In this case the Euler--Lagrange equation for extremals is the following system of ODEs:
 \[
\frac{\delta L}{\delta u^i}:=
\frac{\p L}{\p u^i} -\frac{d}{dt}\frac{\p L}{\p u_1^i} + \frac{d^2}{dt^2}\frac{\p L}{\p u_2^i} -  
\cdots + (-1)^n\frac{d^n}{dt^n}\frac{\p L}{\p u_n^i} =0, 
 \]
where $\tfrac{d}{dt}$ is the operator of total derivative. 
These are obtained from the first variation of the action, and we will not discuss the second variation
and the sufficient condition for critical points to become actual extremals.

While given by a simple formula, the computation of the Euler--Lagrange equations can be hard, depending 
on algebraic complexity of $L$, especially when the order $n$ or the dimension $m$ of the manifold grows. 
However, the application of symmetry and invariance often simplifies straightforward calculations.

Invariant variational problems are natural and often arise in practice. This means 
that the integrand $L\,dt$ is an invariant of a Lie group $G$ action; then 
it can be expressed through differential invariants and hence 
the equation for extremals $\delta_uL=0$ can be also represented 
in invariant terms that simplify the computations. 

This program was initiated by Griffiths \cite{griffiths} (in the context of exterior differential systems) 
and by Anderson \cite{anderson} (via the jet formalism). 
In particular, they derived the following formula in the case of optimal control problems 
for curves on the Euclidean plane.
Denote by $\kappa$ the curvature and by $s$ a natural parameter of the curve,
and let $\kappa_m=\tfrac{d^m}{ds^m}\kappa$.
If $L=L(\kappa,\kappa_s,\kappa_{ss},\dots,\kappa_n)$ is an invariant Lagrangian
and we integrate by $ds$, then the Euler--Lagrange equation is
 \[
(\D_s^2+\kappa^2)\,\mathcal{E}(L)+\kappa\,\mathcal{H}(L)=0,
 \]
where the invariant Euler and Hamilton operators are given by
 \begin{equation}\label{EHformula}
\mathcal{E}(L)=\sum_{m=0}^n(-1)^m\D^m\frac{\p L}{\p\kappa_m},\qquad
\mathcal{H}(L)=\sum_{j<i}(-1)^j\kappa_{i-j}\D^j\frac{\p L}{\p\kappa_i}-L.
 \end{equation}
 
In the work \cite{kogan2001invariant} by Kogan--Olver this was advanced further, in particular a general formula was derived and applied, among other examples, to curves in the Euclidean space. 

The main tool of Kogan--Olver \cite{kogan2003invariant} was the moving frame method 
by Cartan \cite{Cartan}, as extended and generalized in the work of 
Fels--Olver \cite{FO}. While quite powerful, it relies on the assumption of free action 
(which eventually happens for finite-dimensional Lie group actions on the space 
of pure jets, but not on differential equations in general \cite{KS1}) 
and the existence of a (transversal) cross-section of the action 
(which, in general, is true only locally).
Thus their results apply to local differential invariants. 
For instance, in the above case of curves in the (oriented) plane with respect to 
the motion group $G=SO(2)\ltimes\R^2$ the curvature $\kappa$ is not invariant
(it contains radicals and is not algebraic), while $\kappa^2$ is 
a rational differential invariant, cf.\ \cite{kruglikov2016global}.

Based on recent advances in the global generation of the algebra of differential invariants \cite{kruglikov2016global},
we address the problem of globally invariant Euler--Lagrange equations. It turns out that most of
the findings in Kogan--Olver \cite{kogan2003invariant} holds true in this context. Generalizing their result,
we will prove an analogous formula for the invariant Euler--Lagrange equations
via the Euler and Hamilton operators. 
In doing so we not only relax the assumption of free action for the Lie group $G$, but also
allow $G$ to be an infinite-dimensional Lie pseudogroup (equivalently this invariance 
can be treated via the corresponding Lie algebra $\mathfrak{g}$ of vector fields) provided
it is of a moderate size (meaning it allows sufficiently many differential invariants;
more precisely that there is a basis of invariant one-forms in the space of jets).

Then this is applied to several classical problems. For instance, we revisit
the problem of motion group on the Euclidean plane. A globally invariant Lagrangian
can be expressed via global differential invariants, which are $K_0=\kappa^2$
and $K_i=\nabla^i K_0$ for $\nabla=\kappa^{-1}\D_s$. 
Then for a contact-invariant Lagrangian form $L(K_0,\dots,K_n)\omega$, 
where $\omega=\kappa ds=\frac{y_2}{(1+y_1^2)^2}dx$, 
the Euler--Lagrange equation is
 \begin{equation}\label{ABformula}
\mathcal{A}^*\mathcal{E}(L) - \mathcal{B}^*\mathcal{H}(L)=0. 
 \end{equation}
Here the invariant Euler and Hamilton operators are given by the obvious analogs of 
formula \eqref{EHformula} with $\kappa_i\mapsto K_i$, while
\[\mathcal{A}^* = 2 K_0^2\nabla^2 +5 K_0 K_1 \nabla+(2K_0 K_2+K_1^2+2K_0^2), \qquad \mathcal{B}^* = K_0 \nabla^2 +\frac{1}{2} K_1 \nabla.\]

Formula \eqref{ABformula} was derived in \cite{kogan2003invariant} in the general local context. 
We  show that it also works for globally invariant problems and general groups $G$.  
Other examples, considered in this paper, where formula \eqref{ABformula} is applicable, are: 
unparametrized curves in conformal spaceforms (in dimension two, rather M\"obius structure) 
and in flat projective spaces with respect to the corresponding maximal symmetry groups.

We should stress that computations produce intermediate formulas that make it 
difficult to handle even for modern symbolic software and computer capacity. 
We will discuss the advantages and
limitations of the method of moving frames, and show some tricks that allow to 
make these hard computations manageable. 
We will derive both local and global formulas for invariant Euler--Lagrange equations.

In addition, we will discuss regular and singular strata of the action, as the latter
are precisely those where the moving frame does not exist.
We provide an example from conformal geometry (so-called conformal circles,
or conformal geodesics) where the Lagrangian is conformally invariant, yet the
Euler--Lagrange equation forms a singular stratum.

An \textit{invariant variational problem} for unparametrized curves 
$\gamma\subset M$ is defined through a (reparametrization-invariant) functional
 \begin{equation}
S[\gamma] = \int_{\gamma}L\,\omega, 
 \end{equation}
where the \textit{invariant Lagrangian} $L$ is a differential invariant 
and $\omega$ is a contact-invariant horizontal one-form. The corresponding Euler--Lagrange equations 
are invariant and thus, as often stated, ``can be written in terms of differential invariants''. 
This customary refers to ``absolute differential invariants", however from the general principles
one can just conclude that it can be expressed through absolute or relative differential invariants.
The latter occur if the locus of the Euler--Lagrange equations in the space of jets is contained
in the singular stratum of the foliation of this space by $G$-orbits. 

For this purpose a relative invariant factor $W$ should be introduced in expression \eqref{ABformula},
as also explained in \cite{kogan2001invariant}. It was, however, questioned in that paper
whether the Euler--Lagrange equation could happen to land in the singular locus of $W$; in other words,
whether the omission of $W$ as in \eqref{ABformula} is justified.
As mentioned above, the equation of conformal circles provides such an example.
Thus the factor $W$ (which in our approach will be a global relative invariant) is inevitable
in the invariant formulation of the Euler--Lagrange equations.

\subsection*{Structure of the Article}

In the next section, we develop the machinery for invariant variational calculus,
adapted for global differential invariants. We do not rely solely on the technique 
of moving frames; yet, we obtain effective formulas for computations. The method is 
similar to that of \cite{kogan2003invariant} and we relate our approach 
(which does not use invariantizations, central in loc.cit.) to that of Kogan--Olver. 
It was already mentioned in this reference, and  made explicit 
in the following work by Itskov \cite{It}, that computations can be made
explicit for the Lie algebra $\mathfrak{g}$ action via the so-called recurrence 
formulas. 


The matrix relative invariant $W$, an important ingredient hidden in formula 
\eqref{ABformula}, cannot be derived from the cross-section by Lie algebraic
methods. We give a constructive approach to computing it using the symbols of  differential $G$-invariants.
Our main result, Theorem \ref{mainThm}, is the formulation of
the invariant Euler--Lagrange equations via an arbitrary 
invariant coframe, suitable for both global and local approaches.

Then we provide several important examples of computations. References 
\cite{anderson,griffiths,kogan2001invariant,kogan2003invariant} already contained
non-trivial examples. Here we complement those by first deriving formulas
for global invariant problems, that is, by expressing the Euler--Lagrange equations
via global rational differential invariants. In some other examples, not discussed earlier
in the literature, we also provide formulas via local differential invariants, as those may also be of interest.

Thus, we cover 2-dimensional projective and M\"obius geometries, 3-dimensional metric
and conformal geometries, and also 4-dimensional Euclidean (as well as Minkowski) geometries.
We give an explicit general formula for the Euler--Lagrange operator, but also provide
a sample of results for particular Lagrangians depending on jets of curves of low order.

We conclude with an outlook, where we discuss symbolic challenges we faced in the
computations of examples.

\section{Invariant Variational Problems}

\subsection{Jets of curves} \label{sect:jets} 

We shall study jets of unparametrized curves. Given a $n$-dimensional manifold $M^n$ and a 
curve $\gamma: \mathbb{R} \rightarrow M^n$, we may choose coordinates 
$(x,u^1, \dots, u^{n-1})$ on $M$ such that $\gamma = u(x)$ is a graph of 
the function $x \mapsto (u^i(x))$. In this way, we obtain local coordinates for 
the zeroth order jet bundle $J^0(M,1) = M$ of unparametrized curves in 
$M^n$. The $k$-th order jet bundles $J^k(M, 1)$ have coordinates $(x, 
u^i_{j})$ with $0 \leq j \leq k$, where $u^i_{j}$ represents the $j$'th 
derivative of $u^i$, that is, given a curve $\sigma: x \mapsto u(x)$ we have that 
$u^i_j|_{\sigma} = \frac{d^ju^i}{dx^j}(x)$. A curve $\sigma: x \mapsto u(x)$ 
defines its $k$'th jet prolongation $j^k\sigma,\ x \mapsto (x, u^i_j(x))_{0 
\leq j \leq k}$, which is obtained by including all derivatives up to order 
$k$.

\medskip

We have natural projections between jet bundles $\pi_{k}^{k+1}: J^{k+1} \rightarrow J^k$, 
through which we may define the infinite jet bundle $J^{\infty}$ as its inverse limit. 
Accordingly, differential forms on $J^{\infty}$ arise as the pullback of differential 
forms on some finite order jet 
space $J^k$ through the projection $\pi_{k}^{\infty}: J^{\infty} \rightarrow J^k$. We denote by $\Omega^p := \Omega^p(J^\infty(M,1))$ the space of differential $p$-forms on $J^\infty(M,1)$.
\medskip

The contact 1-forms on $J^{\infty}$ are spanned by the following basis   
\begin{equation}\label{def_contact_forms}
    \theta^i_j := du_{j}^i - u_{j+1}^i dx\hspace{1cm} (1\leq i<n,\ j\geq 0).
\end{equation}
Contact forms $\theta$ are uniquely characterized by the property that they 
vanish on the jet prolongation of curves; that is, we have $\theta|_{j^k\sigma}=0$. 
We will refer to $j$ as the \textit{order} of the 1-form $\theta^i_j$ (though it
is supported on $J^{j+1}$). Then we define the order of a general contact 1-form 
$\sum_{j=0}^k\sum_{i=1}^{n-1}a_i^j\theta_j^i$ with $a_i^j\in C^\infty(J^\infty)$
to be $k$ if $a_i^k\neq0$ for some $i$. 
Denote by $\mathcal{C}\subseteq T^{*}J^{\infty}$ the span of all contact forms. 
The dual of $\mathcal{C}$ in $TJ^{\infty}$ is spanned by the \textit{total derivative vector field}
 \begin{equation}\label{def_total_derivative}
D_x = \frac{\p}{\p x} +\sum_{i,j \geq 0} u^i_{j+1}\frac{\p}{\p u^i_j}.
 \end{equation}

\subsection{Invariant differential forms} \label{sect:invariants}

Consider a Lie group $G$ acting on $M=J^0$, and let $\mathfrak{g}=\op{Lie}(G)$ 
be the corresponding Lie algebra of vector fields on $J^0$. The $G$-action on 
$M$ induces an action of $G$ on the space of curves in $M$, and thus on the space 
$J^k$ of $k$-jets of curves. One way to define this canonical prolongation 
$\varphi^{(k)}$ of a diffeomorphism $\varphi \in G$ is by requiring 
$(\varphi^{(k)})^*(\theta_j^i) \in \mathcal{C}$ for $j < k$. 
Similarly, we define the prolongation of a vector field $X$ on $J^0$ by 
the condition $L_{X^{(k)}} \theta_j^i \in \mathcal{C}$ for $j<k$. 

A more general setup is a Lie pseudogroup $G$ of local transformations of $M$
and the Lie algebra of (sheaves of) vector fields $\mathfrak{g}$ that need not
be complete. The theory works in this case as well \cite{kruglikov2016global}.

Since $G$ acts on $J^\infty$, we introduce several types of $G$-invariant objects. A $p$-form $\omega \in \Omega^p$ is $G$-invariant if $(\varphi^{(\infty)})^* (\omega) = \omega$,  and  $\mathfrak{g}$-invariant if $L_{X^{(\infty)}} \omega=0$. (In practice, it is sufficient to use finite indices $k$ in these computations, as $\omega$ is the $\pi_{k}^{\infty}$-pullback of a $p$-form on $J^k$ for some $k$.) If $\omega$ is $G$-invariant, it is also $\mathfrak{g}$-invariant. As a partial converse, if $\omega$ is $\mathfrak{g}$-invariant and globally defined, then it is also $G_0$-invariant, where $G_0\subset G$ is the connected component. 

\begin{definition}
    An invariant function (0-form) is called an absolute differential invariant.  
\end{definition}

It is well-known that the algebra of absolute differential invariants is finitely generated, and that it separates generic orbits. The idea goes back to Sophus Lie, but we will rely on a global version from \cite{kruglikov2016global}. For unconstrained curves, which are our focus, it can be formulated in the following way. 

 \begin{theorem}[The global Lie--Tresse theorem \cite{kruglikov2016global}]
Let $G$ be an algebraic Lie pseudogroup of diffeomorphisms that acts transitively 
on $M$. The field of rational absolute differential invariants separates orbits in 
general position. Furthermore, this field is finitely generated in the sense that 
there exists finitely many rational differential invariants $I^1,\dots,I^r$, 
and a rational invariant derivation $\D$, such that any other rational absolute 
differential invariant is a rational combination of $I^i$ and their derivatives, 
$\D^j(I^i)$. 
 \end{theorem}

An invariant derivation $\D$ has the form $\D = \gamma D_x$, where $\gamma$ is a rational function on $J^k$ for some $k$. If $\alpha$ is a rational invariant 1-form satisfying $\alpha(D_x)\neq 0$, the choice $\gamma = \alpha(D_x)^{-1}$ results in an invariant derivation. If $I$ is a rational absolute differential invariant, the 1-form $\alpha = dI$ leads to a particular invariant derivation, called the Tresse derivative.  We let $\mathfrak{A}$ denote the field of rational differential invariants. 

The algebraicity is a somewhat technical property that usually holds for the transitive actions arising in applications. In particular it holds for all the examples we treat in this paper. 

\begin{remark}
    The adjective \textit{algebraic} refers to the action of the stabilizer $G_a = \{g \in G \mid g(a)=a\}$ for $a \in M$. Due to the transitivity assumption, this property is independent of the point $a$. Assuming transitivity and algebraicity in this way leads to an algebraic group $G_a^k$ acting algebraically on fibers $(\pi_0^k)^{-1}(a)$ of $\pi_0^k\colon J^k \to M$, which are algebraic varieties. Then Rosenlicht's theorem says that the (finitely generated) field of rational invariants separates $G_a^k$-orbits in general position in $(\pi_0^k)^{-1}(a)$. Transitivity on $M$ lets us identify these rational $G_a^k$-invariant functions on $(\pi_0^k)^{-1}(a)$ with absolute differential invariants on $J^k$, and we therefore also refer to the latter as \textit{rational}, even though the dependence on $a$ may be nonalgebraic. Indeed, we make no algebraicity assumptions on the manifold $M$. For further details, we refer to \cite{kruglikov2016global}. 
\end{remark}

Assume now that $G$ is finite-dimensional, and let $I^1,\dots,I^r,\D$ be generators 
for the field of absolute differential invariants, of orders $k_i=\op{ord}(I_i)$. 
Locally one can take $r=n-1$, but globally one may need one more invariant to cover a possible
algebraic extension (one is enough due to the theorem on primitive element).
Thus $n-1\leq r\leq n$ and we may assume that $I^1,\dots,I^{n-1},\D$ generate
the subfield of finite index in the field of all rational differential invariants.

Let $\varpi$ be an invariant one-form satisfying $\varpi(\D)=1$. 
It always exists: for example, choose an invariant $I_0$ with $\D(I_0)\neq0$, then take 
 \begin{equation}\label{eq:varpi}
\varpi=\frac{dI_0}{\D(I_0)}.     
 \end{equation} 
(In particular, if $\D$ is given by the 
Tresse derivative $\D=\D_{I^0}$ then $\varpi=dI^0$.) We assume, without loss of 
generality, that $k_0=\op{ord}(\varpi)\leq\min_{1\leq i<n}[k_i]$. Define the 1-forms
 \begin{equation}\label{eq:varthetatilde}
\tilde\vartheta^i = dI^i - \D(I^i)\varpi, \qquad 1\leq i<n.
 \end{equation} 

 \begin{lemma}
The 1-forms $\tilde\vartheta^1,\dots,\tilde\vartheta^{n-1}$ are $G$-invariant 
and satisfy $\tilde\vartheta^i(\D)=0$. Furthermore, the form $\varpi$ can be 
chosen in such a way that the matrix 
$A_j^i=\tilde\vartheta^i\Bigl(\p_{u_{k_i}^j}\Bigr)$ is nondegenerate at 
generic points of $J^\infty(M,1)$. 
 \end{lemma}

Matrix $A$ represents the symbol of the system of invariants, so
its nondegeneracy is an invariant condition.

 \begin{proof}
The first claim is straightforward. For the second, let $k=\max_{1\leq i<n}[k_i]$
and prolong the given invariants to order $k$:
$\hat{I}^i=\D^{k-k_i}I^i$. Note that if $\D=\gamma D_x$ then
$d\hat{I}^i\Bigl(\p_{u^j_k}\Bigr)=\gamma^{k-k_i}dI\Bigl(\p_{u^j_{k_i}}\Bigr)$.
We note the following 
 $$
\bigl(d\hat{I}^1\wedge\dots\wedge d\hat{I}^{n-1}\wedge\varpi\bigr)
\bigl(\p_{u^1_k},\dots,\p_{u^{n-1}_k},\D\bigr)\neq0.
 $$
Indeed, since $k>k_0$, this inequality follows from the condition of generation of a finite index subfield of differential invariants, discussed before the lemma.
It has an equivalent form
 \begin{equation}\label{dIIpi}
\bigl(d\hat{I}^1\wedge\dots\wedge d\hat{I}^{n-1}\bigr)
\bigl(\p_{u^1_k},\dots,\p_{u^{n-1}_k})\neq0.
 \end{equation}

If $k_0<k_i$ for all $i$, then the claim follows from 
$A^i_j=dI^i\Bigl(\p_{u^j_{k_i}}\Bigr)
=\gamma^{k_i-k}d\hat{I}^i\Bigl(\p_{u^j_k}\Bigr)$ as we conclude that 
$\det(A)$ differs from evaluation in \eqref{dIIpi} by the factor $\gamma^N$, $N=\sum k_i-(n-1)k$.

Let now some invariants $I^1,\dots,I^{\nu-1}$ have minimal order $k_0$
and the rest $I^{\nu},\dots,I^{n-1}$ order $>k_0$. We have invariant splitting
$\langle\p_{u^1_k},\dots,\p_{u^{n-1}_k}\rangle=\hat{\Pi}_0\oplus\hat{\Pi}_1=
\langle d\hat{I}^1,\dots,d\hat{I}^{\nu-1}\rangle^\perp\oplus
\langle d\hat{I}^\nu,\dots,d\hat{I}^{n-1}\rangle^\perp$,
where $\perp$ denotes the annihilator.
Modifying $\varpi$ by $dI^1,\dots,dI^{\nu-1}$ we can achieve 
$\varpi|_{\Pi_1}=0$, where $\Pi_1=\op{ad}_{\D}^{k-k_0}\hat{\Pi}_1$
(that is, $\sum a^j\p_{u^j_{k_0}}\in\Pi_1$ $\Leftrightarrow$ $\sum a^j\p_{u^j_k}\in\hat{\Pi}_1$). Then from \eqref{dIIpi} we obtain:
 $$
\!\!\!\bigl(\D^{k-k_1}\tilde\vartheta^1\wedge\dots\wedge
\D^{k-k_{n-1}}\tilde\vartheta^{n-1}\bigr)
\bigl(\p_{u^1_k},\dots,\p_{u^{n-1}_k})=
\left.\Bigl(d\hat{I}^1\wedge\dots\wedge d\hat{I}^{\nu-1}\Bigr)\right|_{\Lambda^{\nu-1}\Pi_1}\wedge
\left.\Bigl(d\hat{I}^\nu\wedge\dots\wedge d\hat{I}^{n-1}\Bigr)\right|_{\Lambda^{n-\nu}\Pi_0}\neq0
 $$
and the claim $\det(A)\neq0$ follows.
 \end{proof}

Applying $\D$ to $\tilde\vartheta^i$ results in invariant contact 1-forms 
of higher orders. To obtain a complete invariant coframe of $J^\infty$
we proceed as follows. Using coordinates $x,u_j^i$ on $J^\infty(M,1)$, 
define the (non-invariant) contact forms by $\theta_j^i = du_j^i-u_{j+1}^i dx$. 
Then we obtain a new set of $0$th-order contact forms 
$\vartheta_0^i = \sum a_j^i\theta_0^j\in C^\infty(J^\infty(M,1)) 
\otimes_{C^\infty(M)}\Omega^1(J^0(M,1))$ by the requirement 
 \begin{equation}\label{eq:invforms}
\Bigl(\D^{k_i}(\vartheta_0^i)\Bigr)\bigl(\partial_{u_{k_i}^j}\bigr) = 
\tilde\vartheta^i\bigl(\partial_{u_{k_i}^j}\bigr), \qquad 1\leq i,j<n.
 \end{equation} 
Condition \eqref{eq:invforms} determines the coefficients uniquely 
$a_j^i=\gamma^{-k_i}\tilde\vartheta^i\bigl(\p_{u^j_{k_i}}\bigr)$, whence
 \begin{equation}\label{eq:vartheta0}
\vartheta_0^i = \gamma^{-k_i}\sum_{j=1}^{n-1} \tilde\vartheta^i\Bigl(\partial_{u_{k_i}^j}\Bigr) \theta_0^j, \qquad 1\leq i<n. 
 \end{equation}
Even though we defined them in coordinates, the 1-forms $\vartheta_0^i$ are 
actually coordinate independent. Indeed, the action of a diffeomorphism on $\theta_0^l$ is counteracted by its action on $\partial_{u_k^l}$. 

 \begin{lemma}
The 1-forms $\vartheta_0^1,\dots,\vartheta_0^{n-1}$ of order $0$
are independent, $G$-invariant, and satisfy $\vartheta_0^i(\D)=0$. 
 \end{lemma}

 \begin{proof}
The nontrivial part of this statement is the $G$-invariance, but this follows from 
the invariance of condition \eqref{eq:invforms}, which can be rewritten as $\Bigl(\D^k(\vartheta_0^i)\Bigr)\bigl(\partial_{u_k^j}\bigr) = 
\D^{k-k_i}\tilde\vartheta^i\bigl(\partial_{u_k^j}\bigr)$, where we exploit 
$G$-invariance of the subspace $\langle\p_{u^1_k},\dots,\p_{u^{n-1}_k}\rangle
\subset TJ^k$. 
 \end{proof}

Now define $\vartheta_j^i = \D^j(\vartheta_0^i)$, and notice that $\vartheta_j^i(\D)=0$ for every $i$ and $j$. These 1-forms are linearly independent (over $\mathfrak{A}$) at generic points in $J^k$,  and it is clear that they span the same subspace of $T^*J^\infty(M,1)$ as the forms $\theta_j^i$ at these points. We have thus proved the following theorem. 

 \begin{theorem} \label{th:invariantcoframe}
Consider a finite-dimensional algebraic Lie pseudogroup $G$ acting transitively on a manifold $M$. 
Then there exists a global (rational: defined almost everywhere) invariant frame 
of $T^*J^\infty(M,1)$ consisting of one 1-form $\varpi$ satisfying $\varpi(\D)=1$ and infinitely many 
contact forms $\vartheta_j^i$, related by $\D(\vartheta_i^j) = \vartheta_{i+1}^j$. 
 \end{theorem}

\begin{remark}
    An important case where the theorem holds is when $G$ is a Lie group acting on $M$. The assumptions of the theorem are essentially the same as those of the global Lie--Tresse theorem. The finite-dimensionality is a simple way to guarantee that 1-forms \eqref{eq:varthetatilde} are independent. The action also may be given by a global finite-dimensional Lie algebra $\mathfrak{g}$ with algebraic action yet only local Lie group $G$.
    If $G$ or $\mathfrak{g}$ is infinite-dimensional, there will in general not be enough absolute differential invariants to generate an invariant frame of $T^*J^\infty(M,1)$ in this way (yet an invariant coframe may exist). 
\end{remark}

We will also need the notion of relative invariants. A jet function $R \in C^{\infty}(J^{\infty})$ is a \textit{relative invariant} if there 
exists a linear map $w:\mathfrak{g}\to C^{\infty}(J^{\infty})$ such that $L_{X^{(\infty)}}R=w(X) R$ for all $X \in \mathfrak{g}$. We call $w$ the weight of the relative invariant. In 
general there are only finitely many independent weights if properly
interpreted as cohomology classes \cite{KS2}. Relative invariants are 
useful because their vanishing locus is invariant. Consequently, we can use relative 
invariants to understand the structure of singular orbits.

\subsubsection{Curves in the Euclidean plane} \label{sect:euclideanex1}
As a simple example, that we will revisit several times, consider the group of Euclidean transformations on $\mathbb R^2$. The field of rational differential invariants is generated by 
\[ I_0 = \kappa^2 = \frac{y_2^2}{(y_1^2+1)^3}, \qquad \mathcal{D} = \frac{y_1^2+1}{y_2} D_x. \]
We define $I_1 = \D(I_0)$ and $\tilde \vartheta_3 = dI_1-\frac{\mathcal{D}(I_1)}{\mathcal{D}(I_0)} dI_0$, which we use to construct the invariant contact form of order $0$:
\[ \vartheta_0 = \frac{\tilde \vartheta_3(\partial_{y_3})}{\left(\frac{y_1^2+1}{y_2}\right)^3}(dy-y_1 dx) = 2 \frac{y_2^3}{(y_1^2+1)^5} (dy-y_1 dx).\]
The invariant 1-forms are spanned by 
 \[ 
\varpi = \frac{1}{\mathcal{D}(I_0)} dI_0 \quad\text{ and }\quad  
\vartheta_i = \mathcal{D}^i(\vartheta_0),\ i\geq0.
 \]

\subsection{The $\mathcal{C}$-spectral sequence}

In local coordinates, a variational problem for curves is defined through the functional 
\[ S[\sigma] = \int L(x,u^i|_\sigma,u^i_1|_\sigma,\dots) dx.\]
There is some ambiguity in the definition of $L dx$:
\begin{enumerate}[label=(\roman*)]
    \item If $\theta$ is a contact form, then $\theta|_{j^k \sigma}=0$, so the Lagrangian $L(x,u^i,u^i_1,\dots) dx+\theta$ describes the same variational problem. 
    \item If $\alpha$ is an exact 1-form, then the Lagrangian $L(x,u^i,u^i_1,\dots) dx+\alpha$ describes the same variational problem. 
\end{enumerate}
Because of (i), the Lagrangian can be considered an element in $\bar{\Omega}^1 = \Omega^1/\mathcal{C}$. The space $\bar{\Omega}^1$ is part of the horizontal de Rham complex, which for curves is very short: 
 \[ 
0 \to  C^\infty(J^\infty) \xrightarrow{\bar d} \bar{\Omega}^1 \to 0
 \]
The horizontal differential $\bar d$ is defined by $\bar d f = df+\mathcal C$. To define the Euler operator, whose kernel yields the Euler--Lagrange equations, we interpret the horizontal de Rham complex as a small part of a spectral sequence called the $\mathcal{C}$-spectral sequence \cite{V}, which we will now describe. This description is based on \cite[ch. 4.3]{V2}, but indices $p,q$ are swapped to align with those of the variational bicomplex below. 

The space $\mathcal{C}$ of contact 1-forms induces a filtration on $\Omega^\bullet$. If we define 
\[\mathcal{C}^q \Omega^{p+q} := \underbrace{\mathcal{C} \wedge \cdots \wedge \mathcal{C}}_{q\text{ times}} \wedge \Omega^{p} \subset \Omega^{p+q},\]
we see that $\mathcal{C}^q \Omega^{p+q} \subset \mathcal{C}^{q-1} \Omega^{p+q}$, and that $d (\mathcal{C}^{q} \Omega^{p+q}) \subset \mathcal{C}^{q} \Omega^{p+q+1}$. The first two nonzero pages of the spectral sequence associated with this filtration are defined by 
 \[ 
E_0^{p,q} = \mathcal{C}^q \Omega^{p+q} / \mathcal{C}^{q+1} \Omega^{p+q}\simeq
\Lambda^q\mathcal{C}\otimes\Lambda^p\bar{\Omega}^1, 
\qquad E_1^{p,q} = H^{p+q}(\mathcal{C}^q \Omega^{\bullet} / \mathcal{C}^{q+1} \Omega^{\bullet},\bar{d}).
 \]
The exterior differential induces a map 
\[d_r^{p,q} \colon E_r^{p,q} \to E_r^{p-r+1,q+r}.\] In the case of curves, which is our focus, $E_0^{p,q}=0$ holds for $p >1$. The horizontal de Rham complex can now be written 
\[ 0 \to E_0^{0,0} \xrightarrow{d_0^{0,0}} E_0^{1,0} \to 0.\]
Due to points (i) and (ii), the Lagrangian $\lambda=L\,dx$ is naturally associated with an element $[\lambda] \in E_1^{1,0}=H^1(\Omega^\bullet/\mathcal{C}^1 \Omega^\bullet)$. The Euler operator can now be defined in a global, coordinate-independent way as the map 
 \[ 
d_1^{1,0} \colon E_1^{1,0} \to E_1^{1,1}.
 \]

 \begin{remark} \label{rk:spectral}
Note that for curves, only the first three differentials are nontrivial,
 \[
d_0^{p,q} \colon E_0^{p,q} \to E_0^{p+1,q}, \qquad d_1^{p,q}\colon E_1^{p,q} \to E_1^{p,q+1}, \qquad d_2^{p,q}\colon E_2^{p,q} \to E_2^{p-1,q+2},
 \]
and then the spectral sequence stabilizes $E_\infty^{p,q}=E_3^{p,q}$.
 \end{remark}

\subsection{The variational bicomplex}\label{sect:variationalbicomplex}

To see that $d_1^{0,1}[\lambda]=0$ is equivalent to the standard Euler--Lagrange equations for the variational problem, one chooses a coordinate chart of $J^\infty(M,1)$, as described above. This essentially lets us work with $J^\infty(\mathbb R,\mathbb R^{n-1})$ instead of $J^\infty(M,1)$. The kernel of projection $\pi^\infty \colon J^\infty(\mathbb R,\mathbb R^{n-1}) \to \mathbb R$ gives a complement to the subbundle $\mathcal{C} \subset T^*J^\infty$,
\begin{equation}\label{splitting_cotangent}
    T^{*}J^{\infty} = H \oplus \mathcal{C},
\end{equation}
though this complement is not canonical from the viewpoint of $J^\infty(M,1)$. We will refer to 1-forms in $H$ as horizontal 1-forms. In the notation of Section \ref{sect:jets}, $H$ is spanned by $dx$. 
The direct-sum decomposition induces for each $p$ a splitting 
 \[
\Omega^p=\bigoplus_{r+s=p} \Omega^{r,s},
 \] 
where $\Omega^{r,s}$ consists of differential forms of bi-degree $(r,s)$, i.e.,
linear combinations of products of $r$ horizontal and $s$ contact, or vertical, forms. 
(Note that $\Omega^{r,s} = 0$ for $r > 1$ since we only consider curves.) We denote by $\pi_{r,s}$ the projection $\Omega^{\bullet} \rightarrow \Omega^{r,s}$. 

When acting on $\Omega^{r,s}$, the exterior derivative 
$d\colon \Omega^p \to \Omega^{p+1}$ decomposes into a horizontal and a vertical part, $d= d_H+d_V$, with 
\begin{equation}
    d_H \colon \Omega^{r,s} \to \Omega^{r+1,s}, \qquad d_V \colon \Omega^{r,s} \to  \Omega^{r,s+1}.
\end{equation}
The maps $d_H$ and $d_V$ satisfy 
\begin{equation}
     d_H^2=0, \qquad d_V^2=0, \qquad d_H d_V +d_V d_H=0,
\end{equation}
making $(\Omega^{\bullet, \bullet}, d_H, d_V)$ into a bicomplex, called the \textit{variational bicomplex} \cite{anderson}. On jet functions, these differentials are 
defined respectively by
\begin{equation}\label{def_hor_ver_der}
    d_H: \Omega^{0,0}(J^{\infty}) \rightarrow \Omega^{1,0}(J^{\infty}),\quad f \mapsto D_x(f)\ dx,\quad d_V: \Omega^{0,0}(J^{\infty}) \rightarrow \Omega^{0,1}(J^{\infty}),\quad f \mapsto \sum_{i,j } \frac{\partial f}{\partial u_j^i} \theta_j^i 
\end{equation}

\begin{remark}
    On any space of jets of sections of a fiber bundle, the variational bicomplex is uniquely defined because the splitting \eqref{splitting_cotangent} is canonical. 
    In contrast, the variational bicomplex is not a natural structure on the jets of curves (or more general submanifolds). In this case, there is no canonical complement to $\mathcal{C}$, and thus no canonical bigrading on $\Omega^p$. The spectral sequence $E_r^{p,q}$, on the other hand, only requires a filtration on $\Omega^p$ and is therefore well-defined. We will still use the variational bicomplex, and its invariant counterpart described in the next section, as computational tools, as is commonly done.
\end{remark}

To see the relationship between the variational problems and the variational bicomplex, the variational bicomplex is extended in the following way: 
\[\begin{tikzcd}
	&& {} & {} & {} \\
	& 0 & {\Omega^{0,2}} & {\Omega^{1,2}} & {\mathcal{F}^2} & 0 \\
	& 0 & {\Omega^{0,1}} & {\Omega^{1,1}} & {\mathcal{F}^1} & 0 \\
	0 & {\mathbb R} & {\Omega^{0,0}} & {\Omega^{1,0}}
	\arrow[from=2-2, to=2-3]
	\arrow["{d_V}", from=2-3, to=1-3]
	\arrow["{d_H}", from=2-3, to=2-4]
	\arrow["{d_V}", from=2-4, to=1-4]
	\arrow["\pi_*", from=2-4, to=2-5]
	\arrow["{}", from=2-5, to=1-5]
	\arrow[from=2-5, to=2-6]
	\arrow[from=3-2, to=3-3]
	\arrow["{d_V}", from=3-3, to=2-3]
	\arrow["{d_H}", from=3-3, to=3-4]
	\arrow["{d_V}", from=3-4, to=2-4]
	\arrow["\pi_*", from=3-4, to=3-5]
	\arrow["{\delta}", from=3-5, to=2-5]
	\arrow[from=3-5, to=3-6]
	\arrow[from=4-1, to=4-2]
	\arrow[from=4-2, to=4-3]
	\arrow["{d_V}", from=4-3, to=3-3]
	\arrow["{d_H}", from=4-3, to=4-4]
	\arrow["{d_V}", from=4-4, to=3-4]
    \arrow["\mathrm{E}"', from=4-4, to=3-5]
\end{tikzcd}\]

\medskip

Here $\mathcal{F}^i$ is defined as the quotient $\mathcal{F}^i = 
\Omega^{1,i}/d_H\Omega^{0,i}$ of differential forms of degree $(1,i)$ modulo 
divergence terms. In this context the Euler operator is defined by 
 \[ 
\mathrm{E}(\lambda) = \pi_*(d_V \lambda). 
 \]
Any element in $\mathcal{F}^1$ has a unique 
representative of the form $\sum A_i \theta_0^i \wedge dx$. To obtain this representative, the contact forms $\theta_j^i$ with $j \geq 1$ have been removed by repeated integration by parts, i.e., by using the formula 
 \[ 
Fd_V(D_x(u^i_{j-1}))\wedge dx +D_x(F)d_V(u_i)\wedge dx =0\qquad (\op{mod}\,d_{H} \Omega^{0,1})
 \]
to rewrite 
 \[
d_V(Ldx)= \frac{\p L}{\p u_j^i} d_V(u_j^i)\wedge dx =  (-D_x)^j 
\left(\frac{\p L}{\p u_j^{i}}\right)\theta_0^i\wedge dx \qquad (\op{mod}\,d_{H} \Omega^{0,1}).
 \]
Thus the coefficients $A_i$ resemble the classical Euler--Lagrange operator 
(summation over repeated indices)
\begin{equation}
    \mathrm{E}_i(L) = (-D_x)^j(L_{u_{j}^i}). \label{eq:classicalEL} 
\end{equation} 
(The map $\delta$ is related to the Helmholtz conditions, 
see \cite{anderson} for more details.)


\subsection{The invariant variational complex}

The definition of the Euler operator as $d_1^{0,1}$ is coordinate independent, and 
therefore invariant with respect to any group of point transformations. Similarly, 
$\pi_*\circ d_V$ is invariant with respect to any group of fiber-preserving 
transformations. However, the operator $\mathrm{E}_i(L)$ defined by the expression 
 \[ 
d_V(L dx) =\mathrm{E}_i(L)\theta_0^i\wedge dx\qquad (\op{mod}\,d_{H} \Omega^{0,1})
 \]
is not invariant in general. To remedy this, Kogan and Olver  introduced in \cite{kogan2001invariant, kogan2003invariant} an invariant version of the variational bicomplex, and used an invariant version of integration by parts to get an invariant representative of $\pi_* ( d_V\lambda)$. The invariant variational complex is constructed in a similar way to the variational bicomplex above, but by using an invariant coframe instead of $dx, \theta_j^i$. Kogan and Olver relied on the method of moving frames \cite{FO} for constructing the invariant coframe, but the concept of invariant bicomplex does not require the notion of moving frame. In this section we will construct an invariant variational bicomplex from a general invariant coframe, such as the one from Theorem \ref{th:invariantcoframe}. 

We will make the same assumptions about algebraicity and transitivity as in the global Lie--Tresse theorem \cite{kruglikov2016global} recalled in Section \ref{sect:invariants}, and assume that the field of rational differential invariants is generated by 
$I^1,\dots,I^r,\D$. Furthermore, we assume that there exists an invariant basis $\{\varpi, \vartheta_j^i \mid i=1,\dots, n-1, j \geq 0\}$ of 1-forms that span $T^*J^\infty$ at generic points, and that they take the following form: 
\begin{itemize}
    \item $\vartheta_j^i = \sum_{l=1}^{n-1} \sum_{s=0}^j F^i_l \theta_s^l$ are invariant contact forms of order $j$, with $F^i_l$ being rational functions on $J^\infty$.
\item The invariant horizontal 1-form $\varpi$ satisfies $\varpi(\D)=1$. 
\end{itemize}  
A coframe of the kind we constructed from the absolute differential invariants in order to prove Theorem \ref{th:invariantcoframe} is admissible for this purpose,
but any other choice is equally good for what follows. 

We let $\Omega_G^{r,s}$ denote the $G$-invariant $(r+s)$-forms of the form 
 \[ 
\sum A_{i_1 \cdots i_s}^{j_1 \cdots j_s} \varpi^r \wedge \vartheta_{j_1}^{i_1} \wedge \cdots \wedge \vartheta_{j_s}^{i_s},
 \]
where $A_{i_1\cdots i_s}^{j_1 \cdots j_s}$ are invariant functions on 
$J^\infty(M,1)$, $\varpi^0=1$, $\varpi^1=\varpi$, and $r \in \{0,1\}$. We have 
 \[
\Omega_G^p = \bigoplus_{r+s=p} \Omega_G^{r,s},
 \] 
as a vector space over the field of rational absolute differential invariants. 
We have $\Omega_G^p \subset \Omega^p$, but the splittings into 
$\Omega_G^{r,s}$ and $\Omega^{r,s}$ are not necessarily subordinated. 
In general, we have $\varpi\in\Omega^{1,0}\oplus\Omega^{0,1}$ from which it follows that 
 \[ 
\Omega_G^{r,s} \subset \Omega^{r,s} \oplus \Omega^{r-1,s+1}. 
 \]
The exterior derivative of an invariant differential form is invariant, and due to the above inclusion, we have 
 \[ 
d(\Omega_G^{r,s}) \subset \Omega^{r+1,s} \oplus \Omega^{r,s+1} \oplus \Omega^{r-1,s+2},
 \]
and thus 
\begin{equation}\label{eqn_diff_split_inv}
    d(\Omega_G^{r,s}) \subset \Omega_G^{r+1,s} \oplus \Omega_G^{r,s+1} \oplus \Omega_G^{r-1,s+2}.
\end{equation}
In fact, for a general $\varpi$, we have $d\varpi \in  \Omega_G^{1,1} \oplus 
\Omega_G^{0,2}$.  

\begin{remark}
As Kogan and Olver \cite{kogan2001invariant} point out, $\Omega_G^{\bullet,\bullet}$ 
is not a bicomplex in general. However, notice that we can always get a bicomplex by 
choosing appropriate $\varpi$, for example by defining $\varpi$ as in \eqref{eq:varpi}. 
Also in Section \ref{sect:variationalbicomplex} we got a bicomplex because of a particular choice of horizontal form $dx$, or more generally $f(x) dx$. 
What $f(x)dx$ and the 1-form \eqref{eq:varpi} have in common is that their exterior differential is contained in $\Omega^{1,1}$ and $\Omega_G^{1,1}$, respectively. 
\end{remark} 

Let us denote by $\pi_G^{r,s} \colon \Omega_G^{\bullet} \rightarrow \Omega^{r,s}_G$ 
the projection onto the forms of degree  $(r,s)$ with respect to the invariant 
coframe $\varpi, \vartheta^i_j$. In view of \eqref{eqn_diff_split_inv} the exterior derivative splits as $d = d_{\mathcal{H}} + d_{\mathcal{V}} + d_{\mathcal{W}}$, where
 \begin{equation}
d_{\mathcal{H}} = \pi_G^{r+1,s} \circ d: \Omega_G^{r,s} \rightarrow \Omega_G^{r+1,s},\ 
d_{\mathcal{V}} = \pi_G^{r,s+1} \circ d: \Omega_G^{r,s} \rightarrow \Omega_G^{r,s+1},\ 
d_{\mathcal{W}} = \pi_G^{r-1,s+2} \circ d: \Omega_G^{r,s} \rightarrow \Omega_G^{r-1,s+2}.
 \end{equation}
We call $d_{\mathcal{H}}$ and $d_{\mathcal{V}}$ the \textit{invariant horizontal differential} and \textit{invariant vertical differential}, 
respectively. In particular, we have that
\begin{equation} \label{eq:dH}
    d_{\mathcal{H}}(\Omega_G^{r,s}) \subseteq \Omega^{r+1,s}_G \subseteq \Omega^{r+1,s} \oplus \Omega^{r,s+1},
\end{equation}
which shows that $d_{\mathcal{H}}$ agrees with the usual horizontal 
differential $d_H$ up to forms of bi-degree $(r,s+1)$. We will use this property 
to relate the invariant form of Euler--Lagrange equations to the standard form. 
Later, in examples, we demonstrate how to compute the invariant differentials explicitly in practice.

\begin{remark}
    The maps $d_\mathcal{H}, d_\mathcal{V}, d_\mathcal{W}$ are not canonical, but depend on the choice of $\varpi$. The maps $d_0^{p,q}, d_1^{p,q}, d_2^{p,q}$ from Remark \ref{rk:spectral} are the canonical counterparts of $d_\mathcal{H}, d_\mathcal{V}, d_\mathcal{W}$ in the context of the $\mathcal{C}$-spectral sequence. 
\end{remark}

For an invariant Lagrangian\footnote{From now on we will denote by $L$ 
an invariant Lagrangian density (in \cite{kogan2001invariant} they used 
$\Tilde{L}$ for this purpose).}
$\lambda=L\,\varpi\in\Omega_G^{1,0}$ we can now obtain the invariant 
Euler--Lagrange equations by closely following the argument in 
\cite{kogan2001invariant, kogan2003invariant} without referring to 
invariantization operators. Invariance of $\lambda$ is equivalent to $L$ 
being a function of absolute differential invariants, i.e., 
$L=L(I^i,I_1^i,I_2^i,\dots)$, where $I_j^i = \D^j(I^i)$. 
Computing the differential of the $G$-invariant Lagrangian $L\,\varpi$ gives 
\begin{equation}
    d_{\mathcal{V}}(L\,\varpi) =d(L\,\varpi) =\sum_{i,j} \frac{\p L}{\p I^i_j}\ d_{\mathcal{V}}I_j^i\wedge\varpi + L\,d_{\mathcal{V}}\varpi,
\end{equation}
where for the final equality we use that 
$dL=\frac{\p L}{\p I^i_j}\cdot(d_{\mathcal{H}}I^i_j +d_{\mathcal{V}}I^i_j)$ 
together with the fact that 
$d_{\mathcal{H}}I^i_j$ is proportional to the invariant one-form $\varpi$. 
Given the invariants $F,G\in\Omega_G^{0,0}$, one obtains the invariant integration by parts formula \cite{kogan2001invariant, kogan2003invariant}
\begin{equation}\label{eqn_inv_int_by_parts}
    F d_{\mathcal{V}}(\mathcal{D}G) \wedge \varpi = - (\mathcal{D}F)\ d_{\mathcal{V}}G \wedge \varpi - F (\mathcal{D} G) d_{\mathcal{V}}\varpi\quad (\op{mod}\ d_{\mathcal{H}} \Omega_G^{0,1}).
\end{equation}
We use this formula to get an invariant representative for $d_{\mathcal{V}}(L\,\varpi)$ of the form
 \begin{equation}
d_{\mathcal{V}}(L\,\varpi) =\sum_i\mathcal{E}_i(L)d_{\mathcal{V}}I^i\wedge\varpi 
-\mathcal{H}(L)d_{\mathcal{V}}\varpi\quad (\op{mod}\,d_{\mathcal{H}}\Omega_G^{0,1}),
 \end{equation}
where 
 \begin{equation}\label{def_inv_Eulerian_Hamiltonian}
\mathcal{E}_{i}(L):= \sum_{j=0}^{\infty}(-\mathcal{D})^j\frac{\p L}{\p I_j^i}, 
\hspace{1cm} 
\mathcal{H}(L):=\sum_{j>k}I^i_{j-k}(-\mathcal{D})^k\frac{\p L}{\p I^i_j}-L
 \end{equation}
are the \textit{invariant Eulerian} (a vector invariant) and 
\textit{invariant Hamiltonian} (a scalar invariant) respectively. 

\medskip

The invariant forms $d_{\mathcal{V}} I^i$ and $d_{\mathcal{V}} \varpi$ can be expressed in terms of  the basis $\varpi, \D^j(\vartheta_0^i)$ ($j\geq0$),  which means that we can write them in terms of differential operators on the contact forms $\vartheta_0^i$:
\begin{equation}\label{def_inv_ops_A_B}
    d_{\mathcal{V}}I^i = \mathcal{A}_j^i(\vartheta_0^j),\hspace{1cm} d_{\mathcal{V}}\varpi = \mathcal{B}_j(\vartheta_0^j) \wedge \varpi.
\end{equation}
Here $\mathcal{A}=(\mathcal{A}^i_j)$ is an invariant matrix differential operator 
and $\mathcal{B}=(\mathcal{B}_j)$ is an invariant vector differential operator. 
Substituting these expressions into \eqref{def_inv_Eulerian_Hamiltonian} yields
 \begin{equation} \label{eq:dvL}
d_{\mathcal{V}}(L\,\varpi)= \left(\mathcal{E}_i(L)\, \mathcal{A}^i_j(\vartheta_0^j) 
-\mathcal{H}(L)\,\mathcal{B}_j(\vartheta_0^j)\right)\wedge\varpi 
\quad (\op{mod}\,d_{\mathcal{H}}\Omega_G^{0,1}).
 \end{equation}
To get this expression to the desired form, we continue to use integration by parts, but in a slightly different way. For a general invariant vertical 1-form $\vartheta$ we have 
 \[ 
d_{\mathcal{H}} \vartheta = \varpi \wedge \D(\vartheta),
 \]
where $\D( \vartheta)$ is the Lie derivative of $\vartheta$ along $\D$, i.e., $\D(\vartheta):=d(\vartheta(D)) + (d \vartheta)(\D,\cdot) = (d\vartheta)(\D,\cdot)$. It follows that \[ d_{\mathcal{H}}(F \vartheta) =- \D(F \vartheta) \wedge \varpi = -(\D(F) \vartheta+F \D(\vartheta)) \wedge \varpi.\] 
Through this formula, one defines the formal adjoint of $\D \colon \Omega_G^{0,1} \to \Omega_G^{0,1}$ by $\D^*=-\D\colon \Omega_G^{0,0}\to \Omega_G^{0,0}$. By iterating the formula, one defines for a differential operator 
$P=\sum_{i=0}^k p_i\D^i:\Omega_G^{0,1}\to\Omega_G^{0,1}$ its formal adjoint:
 \[ 
P^*:\Omega_G^{0,0}\to\Omega_G^{0,0}, \qquad P^*(F)=\sum_{i=1}^k(-\D)^i(p_iF).
 \]
This lets us move the derivatives away from $\vartheta_0^j$, and over to $\mathcal{E}_i$ and $\mathcal{H}$. The result is
 \begin{equation} \label{eq:dLvarpi}
d(L\,\varpi) =   
\left[\mathcal{A}^*\mathcal{E}(L)-\mathcal{B}^*\mathcal{H}(L)\right]_i 
\vartheta_0^i\wedge\varpi:=  
\left((\mathcal{A}_i^j)^*\mathcal{E}_j(L)-\mathcal{B}_i^*\mathcal{H}(L)\right) 
\vartheta^i_0\wedge\varpi\quad (\op{mod}\,d_{\mathcal{H}}\Omega_G^{0,1}).
 \end{equation}
Note that $(\mathcal{A}_i^j)^{*} =(\mathcal{A}^{*})_j^i$ as the adjoint matrix
consists of transposed adjoint entries.\footnote{The middle part of \eqref{eq:dLvarpi} involves the multiplication of matrix $\mathcal{A}^*$ and vector $\mathcal{E}(L)$ which in \eqref{eq:dvL} was a row, but now should be understood as a column. Then the contraction with $\vartheta_0 \wedge \varpi$ becomes the dot-product. } 

Thus, we obtain a canonical invariant representative for the image of the Euler 
operator, and we can define the invariant Euler--Lagrange equations of the 
Lagrangian $\lambda =L\,\varpi$ or (when $\varpi$ is understood to be fixed) 
of the Lagrangian function $L$ by $\mathrm{E}_{\mathrm{inv}}(L)=0$, where 
 \[ 
\mathrm{E}_{\mathrm{inv}}(L):= \left[\mathcal{A}^*\mathcal{E}(L)- \mathcal{B}^*\mathcal{H}(L)\right].
 \]
To understand how this is related to the classical Euler--Lagrange operator 
\eqref{eq:classicalEL}, we write $L\,\varpi = LR\,(dx+\theta)$ where $\theta$ is 
some vertical 1-form, and recall that the classical Euler operator 
(with respect to the basis $\theta_0^i\wedge dx$) is given by
 \[ 
\mathrm{E}_i(LR) = (-D_x)^j\left(\frac{\partial}{\partial{u_j^i}}(LR)\right).
 \]
Due to the fact \eqref{eq:dH}, we can compare the horizontal parts to get 
 \begin{equation}\label{eq:W}
\vartheta_0^j\wedge\varpi= W_k^j\theta_0^k\wedge dx\quad (\op{mod}\,\Omega^{0,2}). 
 \end{equation} 
where $W$ is a matrix-valued relative differential invariant, and we see that
 \begin{equation}\label{WEL}
\mathrm{E}(LR) = W \cdot \mathrm{E}_{\mathrm{inv}}(L).
 \end{equation} 

In particular, for regular points in $J^\infty$ where $W$ is defined and 
nondegenerate, the classical Euler--Lagrange equations are equivalent to the 
invariant Euler--Lagrange equations $\mathrm{E}_{\mathrm{inv}}(L)=0$. 

We summarize this section in the following theorem.

 \begin{theorem}[\textbf{Invariant Euler--Lagrange Equations}]\label{mainThm}
Consider a finite-dimensional algebraic Lie pseudogroup $G$ acting transitively 
on a smooth manifold $M$.   Let  $I^1,\dots,I^r,\mathcal{D}$ generate 
the field of rational differential invariants on $J^\infty(M,1)$. 

Then there exists an invariant coframe, defined at generic points of 
$J^\infty(M,1)$, which is generated by invariant contact forms $\vartheta_0^i$ 
of order $0$ $(1\leq i<n)$, an invariant 1-form $\varpi$ satisfying 
$\varpi(\mathcal{D}) = 1$, and the invariant derivation $\mathcal{D}$.  
The Euler operator corresponding to an invariant Lagrangian 
 \[
\lambda =L\,\varpi
 \]
takes the form \eqref{WEL}, where matrix $W$ is determined by condition \eqref{eq:W}
and the invariant Euler operator 
${E}_{\mathrm{inv}}(L):= [\mathcal{A}^*\mathcal{E}(L)-\mathcal{B}^*\mathcal{H}(L)]$
is defined in terms of the invariant Eulerian and Hamiltonian \eqref{def_inv_Eulerian_Hamiltonian}
and the differential operators $\mathcal{A}$ and $\mathcal{B}$ determined by \eqref{def_inv_ops_A_B}.
 \end{theorem}

The invariant coframe in the theorem can, for example, be the one defined by \eqref{eq:varpi} and \eqref{eq:vartheta0}. 
Note that in Theorem \ref{th:invariantcoframe} the assumption of finite-dimensional $G$ is to guarantee the existence of an invariant coframe.

 \begin{remark}
The extremals of $L$ are given by the equation ${E}_{\mathrm{inv}}(L)$ in the 
regular region of jet-space, however due to formula \eqref{WEL}, singular extremals 
are supported in the stratum where the rank of $W$ drops, or where $W$ is undefined or 
is singular\footnote{In the case of scalar algebraic $W$, this stratum ($W$-divisor) 
corresponds to zeros and poles of $W$; in some cases they are actual extremals. 
An example is given by conics and straight lines for 
the projective parameter considered in the next section.}. 
This relative invariant was introduced in \cite{kogan2001invariant} in the 
context of moving frames, where a question about whether it is essential was raised. 
We show with the example of conformal geodesics below that its vanishing may 
actually give all extremals.

Note however that due to definition \eqref{eq:W} we have the freedom of multiplying $W$
by an absolute differential $I$ by changing the basis element $\varpi$ to $I \varpi$. This may change the zeros and poles of $W$, yet this is just 
a factor exchange in the product \eqref{WEL} given by a matrix of absolute differential invariants, and so should be rather considered as part of 
the equation ${E}_{\mathrm{inv}}(L)$ with a proper choice of generators.
\end{remark}

\subsubsection{Curves in the Euclidean plane}\label{S251}

Let us continue with the example from Section \ref{sect:euclideanex1}. We use
 \[ 
I_0 = \frac{y_2^2}{(y_1^2+1)^3}, \qquad \varpi = 
\frac{y_2}{(y_1^2+1)^2} (dx+y_1 dy), \qquad \vartheta_0 = 
\frac{y_1^2+1}{y_2} (dy-y_1 dx), \qquad \mathcal{D} = \frac{y_1^2+1}{y_2} D_x, 
 \]
as generators for the field of rational absolute differential invariants and for the rational invariant 1-forms. Define
 \[
I_i = \D^i(I_0), \qquad \vartheta_i = \D^i(\vartheta_0).
 \]
A direct symbolic computation in coordinates leads to 
\begin{align*}
    \mathcal{A}(\vartheta_0) &= d_{\mathcal{V}} I_0 = \left(2 I_0^2\D^2 +3 I_0 I_1 \D+I_0(I_2+2 I_0) \right)(\vartheta_0),\\
    \mathcal{B}(\vartheta_0) &= - i_\D d\varpi = \left(I_0 \D^2+\frac{3}{2} I_1\D+ \frac{1}{2} I_2\right)(\vartheta_0),
\end{align*}
and thus the Euler--Lagrange equation is given by 
$\mathcal{A}^*\mathcal{E}(L)-\mathcal{B}^*\mathcal{H}(L)=0$, where
 \begin{equation}\label{EHeqEucG}
\mathcal{A}^* = 2 I_0^2\D^2 +5 I_0 I_1 \D+(2I_0 I_2+I_1^2+2I_0^2), \qquad \mathcal{B}^* = I_0 \D^2 +\frac{1}{2} I_1 \D.
 \end{equation}
Notice that if we instead choose $\varpi = \D(I_0)^{-1} dI_0$ from $\eqref{eq:varpi}$, the differential operator 
$\mathcal{A}$ will vanish, but the expression for $\mathcal{B}$ will become more complicated. 
 
In particular, for the Lagrangian $L=I_0^m$ we obtain from formula \eqref{EHeqEucG}
the equation ${E}_{\mathrm{inv}}(L)=0$, which modulo the trivial factor $K_0$ is equivalent to
 \begin{equation}\label{ELeuclGlob}
(m+\tfrac12)I_0 I_2 + (m^2-\tfrac14) I_1^2+I_0^2=0.
 \end{equation}

\subsection{Computing invariant Euler--Lagrange equations with moving frames} \label{sect:movingframe}

Let us recall the moving frame method and its application to the invariant 
variational complex following Kogan--Olver \cite{kogan2001invariant},
\cite{kogan2003invariant}. 
We emphasize that this setup, whenever applicable, 
is extremely convenient for computations since explicit formulas for 
the invariant differentials can be obtained  from
the corresponding Lie algebra action.

\medskip

Consider a transitive Lie group action of an $l$-dimensional connected Lie group $G$ 
on $M$. By prolongation, we obtain an action $\psi: G \times 
J^{\infty} \rightarrow J^{\infty}$ of $G$ on $J^{\infty}(M,1)$ which is free on an 
open dense subset $\mathcal{O} \subseteq J^{\infty}$ \cite{AO}. 
We describe how to invariantize differential forms on $J^{\infty}$ using 
a moving frame. Recall that a moving frame on $J^{\infty}$ for a free $G$-action 
is a $G$-equivariant map $\rho: J^{\infty}\supseteq\mathcal{O}\rightarrow G$. 
Such a moving frame is equivalent to a local cross-section 
$\Sigma\subseteq\mathcal{O}$ transversal to the orbits of the $G$-action. 
Indeed, the moving frame assigns to a point $z \in \mathcal{O}$ the group 
element in $G$, which transforms $z$ to the unique point lying in the intersection 
$\text{Orbit}(z)\cap\Sigma$ of the orbit with the cross-section. 
Conversely, the cross section can be obtained by solving for the points in 
$J^{\infty}$ on which the moving frame becomes the identity element. 

\begin{definition}[\textbf{Invariantization}]
Let $\alpha \in \Omega^{\bullet}(J^{\infty})$ be a differential form. The \textit{invariantization} $\iota(\alpha)$ is the $G$-invariant differential form on $J^{\infty}$
\begin{equation}\label{def_invariantization}
    \iota(\alpha): z \mapsto (\psi_{\rho(z)}^{*}\alpha)_z
\end{equation}
obtained by pulling back $\alpha$ pointwise along the action $
\psi_{\rho(z)}: J^{\infty} \rightarrow J^{\infty}$ of the moving frame group element.

\medskip

(\textbf{NB}: The invariantization $\iota(\alpha)$ is \textit{not} the 
same as the pullback of $\alpha$ along $\psi_{\rho}: J^{\infty} 
\rightarrow J^{\infty},\ z \mapsto \rho(z) \cdot z $. In particular, 
when differentiating the invariantization, one also differentiates the 
group parameters and consequently the differential $d$ and 
invariantization $\iota$ do not commute.)
\end{definition}

In this way, we obtain the invariantization map 
$\iota: \Omega^{\bullet} \rightarrow \Omega_G^{\bullet}$ 
from the differential forms to the $G$-invariant differential forms on $J^{\infty}$. 
The invariantization of jet coordinates yields differential invariants of 
the same order or constants (phantom invariants) coming from restriction 
of jet-coordinates to the cross-section. 

 \begin{remark}\label{Rk26}
In general, the orbit foliation does not admit a global section, 
but only a quasi-section (intersecting every orbit finitely many times). 
For this  reason, the output of invariantization is only a local invariant.
In fact, the invariantization of the restriction of a local differential invariant  
to the cross-section is the same invariant, locally near the ``section''. 
Yet, the invariantization of the restriction of a global 
differential invariant is the same invariant globally.
Thus, the moving frame method can also be applied to compute global invariants:
the functions on the cross-section giving rise to global differential invariants
via invariantization, form a subfield of the field of all rational functions.
 \end{remark}

Invariantizing the coframe 
$(dx, \theta^i_j)$ gives an invariant coframe on $J^{\infty}$. We say that 
$\varpi := \iota(dx)$ is the \textit{invariant horizontal form} whereas 
$\vartheta^i_j := \iota(\theta^i_j)$ are the \textit{invariant contact forms} 
of order $j$. As before, the invariant coframe induces a decomposition 
$\Omega_G^{\bullet} = \oplus_{r,s}\ \Omega_G^{r,s}$ where each summand consists of 
linear combinations of $r$ invariant horizontal forms and $s$ invariant contact forms. We let $\mathcal{D}$ denote the invariant derivation given by $\varpi(\mathcal{D})=1$. 

The exterior derivative $d: \Omega_G^{\bullet} \rightarrow \Omega_G^{\bullet}$ splits 
accordingly as $d = d_{\mathcal{H}} + d_{\mathcal{V}} + d_{\mathcal{W}}$. By using 
the moving frame method, one can explicitly compute these invariant differentials. 
To this end, choose a basis $X_1,\dots,X_l$ for the Lie algebra $\mathfrak{g}$ 
of $G$ and use this basis to decompose the Maurer--Cartan form 
$\mu \in \Omega^1(G, \mathfrak{g})$ into real-valued forms 
$\mu^1, \dots, \mu^l \in \Omega^1(G)$. Decompose the pullback 
$\rho^{*}\mu^k = \gamma^k + \epsilon^k \in \Omega^1(\mathcal{O})$ into its invariant horizontal part 
$\gamma^k$ and its invariant contact part $\epsilon^k$. 
Let $\alpha \in \Omega^{r,s}$ be a differential form of bi-degree $(r,s)$. 
Then the invariant horizontal differential $d_{\mathcal{H}}: \Omega_G^{r,s} \rightarrow \Omega_G^{r+1,s}$ can be computed by the formula
 \begin{equation}\label{def_inv_hor_der}
d_{\mathcal{H}}\iota(\alpha) = \iota(d_H\alpha) + \sum_{k = 1}^{l} \gamma^k \wedge \iota(\pi_{r,s}\left(X_k(\alpha)\right)),    
 \end{equation}
and the invariant vertical differential $d_{\mathcal{V}}: \Omega_G^{r,s} \rightarrow \Omega_G^{r,s+1}$ can be computed by the formula
 \begin{equation}\label{def_inv_vertical_der}
d_{\mathcal{V}}\iota(\alpha) =   \iota(d_V\alpha) + \sum_{k = 1}^{l} \epsilon^k \wedge \iota(\pi_{r,s}\left(X_k(\alpha)\right)) +  \gamma^k \wedge \iota(\pi_{r-1,s+1}\left(X_k(\alpha)\right)).  
 \end{equation}
For jet-functions $f \in C^{\infty}(J^{\infty})$, the invariant horizontal 
and invariant vertical differentials reduce to
 \begin{equation}\label{eqn_inv_der_on_functions}
d_{\mathcal{H}}\iota(f) = \iota(d_Hf)+\sum_{k = 1}^l \iota(X_k(f))\cdot\gamma^k, \hspace{1cm} 
d_{\mathcal{V}}\iota(f) =\iota(d_Vf)+\sum_{k = 1}^l \iota(X_k(f))\cdot\epsilon^k.
 \end{equation}
In equations \eqref{def_inv_hor_der} - \eqref{eqn_inv_der_on_functions}, $X_k(\alpha)$ and $X_k(f)$ denote the Lie derivatives along $X_k$ understood as a vector field on $M=J^0$ prolonged to $J^\infty$. 

In practice, we can compute the invariant Euler--Lagrange equations using the 
following scheme. Fix a cross-section $\Sigma\subseteq\mathcal{O}$ by 
setting $l=\op{dim}(G)$ each jet-coordinates equal to some constant (care must be taken to ensure that the cross-section intersects generic orbits). 
Obtain a generating set of scalar differential invariants $I^i$
by invariantizing suitable jet functions
(note that the order of $I^i$ will be the same as the order of the chosen jet function). 
The rest is formulated as an algorithm below.

\begin{table}[h!]
    \centering
\scalebox{1}{
\begin{tcolorbox}
\textbf{Algorithm. (Computing invariant Euler--Lagrange eqs using cross-section)}.
\newline
(\textbf{Input}: Basis $X_i$ of Lie algebra $\mathfrak{g}\subseteq\mathfrak{X}(J^\infty)$, cross-section $\Sigma\subseteq\mathcal{O}$ transversal to regular orbits)
\begin{itemize}
    \item Step 1.) Apply  $d_{\mathcal{H}}$ and $d_{\mathcal{V}}$ to the $l$ jet coordinates that are constant on the cross-section.
    
    \item Step 2.) Solve this linear system for $\gamma^k$ and $\epsilon^k$. (So we have computed the pullbacks of Maurer--Cartan forms under moving frame $\rho: \mathcal{O} \rightarrow G$.) 

    \item Step 3.) Compute invariant vertical differentials $d_{\mathcal{V}}I^i$ and $d_{\mathcal{V}}\varpi$ using \eqref{def_inv_vertical_der}.
    
    \item Step 4.) Compute invariant horizontal differential of invariant contact forms $
    \vartheta^i_j$ using \eqref{def_inv_hor_der}. Note that the first term in \eqref{def_inv_hor_der} is given by $\varpi \wedge \vartheta_{j+1}^i$.

    \item Step 5.) Using the fact that $d_{\mathcal{H}}\vartheta^i_j = 
    \varpi \wedge \mathcal{D}(\vartheta^i_j)$ and the previous 
    step, express each invariant contact form $\vartheta^i_j = \sum P_l^i(\vartheta^l_0)$ 
    as the sum of actions of invariant differential operators $P_l^i$ on the invariant 
    contact forms of order $0$.

    \item Step 6.) Rewrite $d_{\mathcal{V}}I^i$ and $d_{\mathcal{V}}\varpi$ in terms of actions of invariant DOs on invariant contact forms of order 0. From these we get the invariant DOs $\mathcal{A}, \mathcal{B}$ as in \eqref{def_inv_ops_A_B}.

    \item Step 7.) Compute the formal adjoints $\mathcal{A}^{*}$ and $\mathcal{B}^{*}$. 
    Invariant EL are then computed using the formula $\mathcal{A}^{*}\mathcal{E}(L) - 
    \mathcal{B}^{*}\mathcal{H}(L) = 0$. (Valid on a regular neighborhood of $\Sigma$.)

\end{itemize}
(\textbf{Output}: Formula for invariant EL eqs in terms of abstract 
invariants $I^i$ and $\mathcal{D}$.)
\end{tcolorbox}}\label{algorithm}
\end{table}

We emphasize that the invariants $I^i$ are not given by explicit expressions if we 
only have a cross-section. Thus the resulting invariant EL are also given in terms of these abstract invariants (given as invariantizations of certain jet functions).  In order to get explicit expressions, there are a couple ways we can 
proceed. First, we can try to compute the explicit moving frame, but this is 
computationally difficult as it requires prolonging the group action. Also the group action may not be given explicitly, for example if the starting point is only a Lie algebra of vector fields.  Secondly, if we already
have explicit expressions for the differential invariants, we can relate these 
to the abstract invariants coming from invariantization by 
restricting to the cross-section and compare to the (restriction of) invariants $I^i$.
Furthermore, we can also use a different invariant coframe such as in Section 
\ref{sect:invariants}, and relate these to the invariantized coframe. 
This latter approach is particularly helpful for computing $W$ 
because in this case it is crucial to have explicit expressions for the 
invariant contact forms of order $0$. We demonstrate this with the 
example of conformal geodesics in 
\S\ref{singular_extremals_conf_geod}.

\subsubsection{Curves in the Euclidean plane}

Let us revisit the example from Section \ref{S251} with the moving frame approach.
We may use the cross-section determined by the vanishing of $x,y,y_1$. Then the invariantizations of $x,y,y_1$ vanish by construction, whereas 
\begin{equation}
    \iota(y_2) = \kappa := \frac{y_2}{(1 + y_1^2)^{3/2}}.
\end{equation}
is the curvature. The higher-order differential invariants are obtained by applying the arclength total derivative $D_s = \frac{1}{\sqrt{1 + y_1^2}} D_x$ to $\kappa$, let us denote $\kappa_i = D_s^i(\kappa)$. The invariant horizontal form is given by
\begin{equation}
    \varpi = \iota(dx) = \frac{dx + y_1 dy}{\sqrt{1 + y_1^2}}.
\end{equation}
Finally, one obtains that the invariant vertical differentials of $\kappa$ and $\varpi$ are given by
\begin{equation}
    d_{\mathcal{V}}\kappa = \iota(\theta_2) =  (D_s^2 + \kappa^2) \iota(\theta_0), \qquad d_{\mathcal{V}}\varpi = - \kappa\ \iota(\theta_0) \wedge \varpi, 
\end{equation}
so that $\mathcal{A}^{\kappa} = D_s^2 + \kappa^2$ and $\mathcal{B} = - \kappa$. Both invariant differential operators are formally self-adjoint so the invariant Euler--Lagrange equations for an invariant Lagrangian $L\, ds$ become
\begin{equation}\label{ELeuclLoc}
    (D_s^2 + \kappa^2) \mathcal{E}(L) + \kappa \mathcal{H}(L) = 0,
\end{equation}
where $\mathcal{E}(L), \mathcal{H}(L)$ are the invariant Eulerian and 
invariant Hamiltonian (\ref{def_inv_Eulerian_Hamiltonian}), respectively.

\medskip

Let us now re-derive the formulas from Section \ref{S251} through the moving frame 
as indicated in Remark \ref{Rk26}. The global horizontal form is $\kappa\, ds$
and the Lagrangian density is $L=I_0^m=\kappa^{2m}$. Thus, the Lagrangian
can be re-written in local invariants as $L\, \kappa\, ds=\kappa^{2m+1}ds$ 
and then from \eqref{ELeuclLoc} with $L\mapsto\kappa^{2m+1}$ 
we obtain the Euler--Lagrange equation (again omitting the trivial factor $\kappa$)
 \begin{equation}\label{ELeuclLoc2}
(2m+1)\kappa\,\kappa_{ss} +(4m^2-1)k_s^2 +\kappa^4 =0.
 \end{equation}
This formula converts to \eqref{ELeuclGlob} upon relation of the invariants:
$\kappa=\sqrt{I_0}$, $\kappa_s=\tfrac12 I_1$, $\kappa_{ss}=\tfrac12\sqrt{I_0}I_2$.

\subsection{Global invariants via cross-section}\label{S27}

An important property of the invariantization operator $\iota$ of \eqref{def_invariantization} defined via a cross-section $\Sigma$ is the following: 
if $J$ is a local differential invariant (function or form or other geometric object), 
then in its domain of definition (open but not necessarily dense in the jet-space) we have: 
 \begin{equation}\label{JKJ}
 \iota(J)=J.
 \end{equation}
This may be violated globally (as in the example of the Euclidean group 
action on jets of curves in the plane, where $J=\kappa$  becomes multi-valued upon analytic continuation). Yet, whenever $J$ is a global differential invariant, 
equation \eqref{JKJ} holds on an open dense set.
Indeed, a rational function on the space of jets is uniquely defined by its restriction to any open set.

Note that invariantization of $J$ is determined by evaluation at the points 
of the cross section $\Sigma$, denoted $J|_\Sigma$ 
(the same as restriction for functions but richer for differential forms)
and then \eqref{JKJ} may be written more instructively as $\iota(J|_\Sigma)=J$.
Normalization of the moving frame (evaluation at $\Sigma$) is an effective way to compute 
in the calculus of differential invariants, as it reduces the number of parameters.

 \begin{theorem}
If $L\,\varpi$ is a globally invariant Langrangian form with respect to a Lie group $G$, 
then its Euler--Lagrange equation $\mathrm{E}_{\mathrm{inv}}(L)$ is also globally $G$-invariant.
 \end{theorem}

The proof follows from the fact that all ingredients of the Euler operator  
(linearization and formal adjoint, see \cite{V}) are manifestly covariant, 
and hence invariant with respect to symmetry.
A few remarks are of order.

\medskip

1.  In the global algebraic setting, the components of the factor $W$ in $\mathrm{E}(LR)$ are rational, but not $G$-invariant in general.  This is illustrated with examples of 
extremals for projective length in \S \ref{subsecproj} and 
conformal geodesics in \S \ref{subsConf}. However considered as a section of a
$G$-equivariant vector bundle, as in \cite{KS2}, it actually becomes invariant. 

2. The method of moving frames can be applied to globally invariant Lagrangians.
Indeed pass from $L\,\varpi$ to its representation via local differential invariants
obtained by invariantization $L'\,\varpi'$. Then $L'=LR$ for the density 
$R=\varpi(D_{\varpi'})$  is a local invariant, and the resulting Euler operator
$\mathrm{E}_{\mathrm{inv}}(L)$ is $\mathfrak{g}$-invariant for 
$\mathfrak{g}=\op{Lie}(G)$. In fact, it is $G$-invariant, though written through 
a basis of local differential invariants.

3. It is convenient to use the more general setup of Theorem \ref{mainThm}
in combination with normalizations of a moving frame to compute the
invariant Euler operator $\mathrm{E}_{\mathrm{inv}}(L)$.
This will be demonstrated in the following computations.

\section{Computations for classical geometries}

In this section we compute the invariant Euler--Lagrange equations for 
several Lie group actions. In $2D$ we consider the  action of 
the M\"obius group $\text{PSL}(2,\C)=\text{SO}(1,3)$
as well as the projective linear group $\text{PSL}(3, \mathbb{R})$. 
We do these computations using both local and global differential invariants. 
In $3D$ we consider the group action of the Euclidean motion group $\text{SE}(3)$ 
as well as the conformal group $\text{Conf}(3)=\text{SO}(1,4)$. 
The general formula for invariant Euler--Lagrange equations for 
$\text{SE}(3)=\text{SO}(3)\ltimes\R^3$ was already computed in 
\cite{anderson}, \cite{griffiths}, \cite{kogan2003invariant}. 
We revisit it and compare the results with the conformally invariant 
Euler--Lagrange equations in the case of conformal geodesics. 
Additionally, we derive a formula for $\text{SE}(3)$ using global invariants. 
Finally, we also compute the invariant Euler--Lagrange equations for the $4D$ 
motion group $\text{SE}(4)$. 
For the computations we used the DifferentialGeometry package in Maple.
See the Maple files, which are supplementary to the arXiv submission.

\subsection{M\"obius Invariant Variational Problems in 2D}

The group $G=\text{SO}(1,3)$ acts (algebraically) on $\mathbb{S}^2$ by conformal transformations. 
It has double cover $\text{Spin}(1,3)\simeq \text{SL}(2,\C)$ with center $\mathbb{Z}_2$,
the quotient by which is $G$. Sometimes this double cover is called the M\"obius group, even though it acts on 
$\C P^1\simeq\mathbb{S}^2$ through $\text{PSL}(2,\C)\simeq  \text{SO}(1,3)$.
The corresponding Lie algebra is $\mathfrak{g}=\mathfrak{so}(1,3)$ 
and, in a chart $\R^2(x,y)$, it has generators 
$\p_x$, $\p_y$, $x\p_y-y\p_x$, $x\p_x+y\p_y$, 
$\frac12(x^2+y^2)\p_x-x(x\p_x+y\p_y)$, $\frac12(x^2+y^2)\p_y-y(x\p_x+y\p_y)$, 
which we denote by $X_1, \dots, X_6$, respectively.

\subsubsection{Generating differential invariants}

It is straightforward to compute that local differential invariants are generated by one 
differential invariant of order 5 and one invariant derivation,
both expressed through the classical curvature $\kappa$ and the natural parameter $ds$ of 
the curve:
 \begin{equation}
q=\frac{4\kappa_s\kappa_{sss}-4\kappa^2\kappa_s^2 -5\kappa_{ss}^2}{\kappa_s^3},\qquad
\sigma=\sqrt{\kappa_s}ds,\quad D_\sigma= \frac1{\sqrt{\kappa_s}}\frac{d}{ds}.
 \end{equation}
In other words, any differential invariant can be described locally as a function of $q$, $q_\sigma$, $q_{\sigma\sigma}$, $q_{3\sigma}=q_{\sigma\sigma\sigma}$, $q_{4\sigma}$ and so on.

\medskip

To obtain global invariants, observe that $q$ is a rational differential invariant. 
However the invariant derivation contains radicals. As an example of a rational invariant derivation, we have
 $$
\bar\nabla=q_\sigma^{-1}D_\sigma,
 $$
which coincides with the Tresse derivative $D_q$. As the Tresse derivative, it has the property $D_q(q)=1$, hence 
cannot generate the algebra of global invariants together with $q$ alone. To compensate for this we have to add the invariant $(q_\sigma)^2$ of order 6.
Thus the field of global M\"obius invariants is generated by
 $$
q,\quad (q_\sigma)^2,\quad \bar\nabla.
 $$
Alternatively one can choose the smaller set of global generators
 \begin{equation}\label{globqq}
q,\quad \nabla := q_\sigma D_\sigma.
 \end{equation}

\subsubsection{M\"obius-invariant Euler--Lagrange equations via local invariants} 
As a cross-section we take
\begin{equation}\label{Kmoeb}
    \Sigma := \{ x = 0, y = 0, y_1 = 0, y_2 = 0, y_3 = 1, y_4 = 0 \}.
\end{equation}
Restricting the M\"obius curvature $q$ to the above cross-section gives $q|_\Sigma = 4 y_5$ so 
that $q$ is obtained as the invariantization of $4y_5$. Moreover, we have that $\sigma|_\Sigma = 
dx$, so we use $\varpi := \iota(dx)$. This ensures that  $\varpi$ 
coincides with $\sigma$ up to a contact form. By application of the algorithm, we obtain 
\begin{equation*}
    \begin{split}
\gamma^1 & = -\varpi,\ \gamma^2 = 0,\ \gamma^3 = 0,\ \gamma^4 = 0,\ \gamma^5 = -\frac18 q \varpi,\ \gamma^6 = -\varpi, \\
\epsilon^1 & = 0,\ \epsilon^2 = -\iota(\theta_0),\ \epsilon^3 = -\iota(\theta_1),\ \epsilon^4 = \frac12 \iota(\theta_3),\ \epsilon^5 = -\frac12 \iota(\theta_4),\ \epsilon^6 = -\iota(\theta_2).
    \end{split}
\end{equation*}
Knowing $\gamma$- and $\epsilon$-components, we compute next the invariant vertical differentials $d_{\mathcal{V}}q$ and $d_{\mathcal{V}}\varpi$. In order to read off $\mathcal{A}, \mathcal{B}$ we then rewrite all invariant contact forms in terms of the action of an invariant differential operator acting on the lowest order invariant contact forms (steps 4-6 in the algorithm). 
The final formula for the invariant Euler--Lagrange equation is given by
\begin{equation}\label{AEBH-Moeb}
    \begin{split}
\mathcal{A}^*\mathcal{E}(L)- \mathcal{B}^*\mathcal{H}(L) 
  = & \Bigl(-4 D_{\sigma}^5+ 2qD_{\sigma}^3+ \tfrac72q_{\sigma}D_{\sigma}^2
+(3q_{\sigma\sigma}-\tfrac14q^2-16)D_{\sigma}+ (q_{3\sigma}-\tfrac38qq_{\sigma})\Bigr)\,\mathcal{E}(L) \\ 
    & + \Bigl(\tfrac12D_{\sigma}^3-\tfrac18qD_{\sigma}-\tfrac1{16}q_{\sigma}\Bigr)\,\mathcal{H}(L),
    \end{split}
\end{equation}
where $D_{\sigma}$ is the M\"obius-invariant derivation (dual to $\varpi$), 
and $\mathcal{E}(L), \mathcal{H}(L)$ are the Eulerian and Hamiltonian of a 
M\"obius-invariant Lagrangian. By application of formulas (\ref{eq:varthetatilde}), (\ref{eq:vartheta0}) and evaluating on the cross-section $\Sigma$, we obtain that the invariantization of $\theta_0 = dy - y_1 dx$ is given by
\begin{equation}
    \iota(\theta_0) = \frac{dq(\partial_{y_5})}{4 D_{\sigma}(x)^5} \cdot \theta_0.
\end{equation}
Using this expression and the definition (\ref{eq:W}) we compute 
\begin{equation}\label{eq:W_Moebius}
    W = \frac{y_1^2 y_3 - 3 y_2^2 y_1 + y_3}{(y_1^2 +1)^3}.
\end{equation}
Solving $W = 0$ for $y_3$ yields the equation for unparametrized circles in $\mathbb{R}^2$, namely
\begin{equation}
    y_3 = \frac{3 y_1 y_2^2}{y_1^2 + 1}.
\end{equation}
We conclude that the singular extremals are precisely all circles
in the plane (including infinite radius: straight lines).

\medskip

We compute some examples of M\"obius-invariant Euler--Lagrange equations 
 (for regular extremals)
by applying formula \eqref{AEBH-Moeb} to the M\"obius-invariant Lagrangians 
$q^n\,\sigma$ for $n=0,1,2,3$. The results are displayed in the following table.

\bigskip

\makegapedcells
\begin{tabular}{||c c c c||} 
 \hline
 Lagrangian $L$ & $\mathcal{E}(L)$ & $\mathcal{H}(L)$ & Invariant Euler--Lagrange Equations \\ [0.5ex] 
 \hline\hline
$\sigma$ & $  0 $ & $ -1 $ &  $q_{\sigma} = 0$  \\ 
 \hline
$q\, \sigma$ & $ 1 $ & $ -q $ &  $q_{3\sigma} = \frac{3 q q_{\sigma}}{8}$  \\ 
 \hline
$q^2\sigma$ & $ 2q $ & $ -q^2 $ & $q_{5\sigma} = \frac58qq_{3\sigma} 
+\frac54 q_{\sigma}q_{\sigma\sigma} -\frac{15}{128}q^2q_{\sigma}-4q_{\sigma}$ \\ 
 \hline
$q^3\sigma$ & $ 3q^2 $ & $-q^3 $ &  $q_{5\sigma} = \frac{P(q,q_{\sigma}, q_{\sigma\sigma},q_{3\sigma},q_{4\sigma})}{384q} $ \hspace{0.5cm}($P$ polynomial)  \\ 
 \hline
\end{tabular}

\medskip

Here the polynomial $P$ is given by the following formula
\begin{equation*}
    \begin{split}
P & = -1920q_{\sigma}q_{4\sigma}-3840q_{\sigma\sigma}q_{3\sigma}+216q^2q_{3\sigma}
+1056qq_{\sigma}q_{\sigma\sigma}+288q_{\sigma}^3-35q^3q_{\sigma}-1536qq_{\sigma}.
    \end{split}
\end{equation*}

\subsubsection{M\"obius-invariant Euler--Lagrange equations via global invariants}

We take generators $q$ and $\nabla=q_{\sigma}D_{\sigma}$ as in \eqref{globqq} and 
let $q_1 :=\nabla q=q_{\sigma}^2$, $q_2 :=\nabla q_1=2q_{\sigma}^2q_{\sigma\sigma}$, 
etc. Let $\Delta:=q_{\sigma}^{-1}\sigma$ be the (horizontal) form dual to $\nabla$. 
We use the cross-section $\Sigma$ from \eqref{Kmoeb}, in which case we have
$\Delta|_\Sigma=\frac1{4y_6}dx$. 
Denote $\varpi_2:=\iota\left(\frac1{4y_6}\,dx\right)=\frac{\varpi}{q_{\sigma}}$. 
Note that
\begin{equation*}
    \begin{split}
        dq & = d_{\mathcal{H}}q + d_{\mathcal{V}}q = D_{\sigma}(q)\, \varpi + d_{\mathcal{V}}q 
         = \nabla(q)\, \varpi_2 + d_{\mathcal{V}}q.
    \end{split}
\end{equation*}
Since $\varpi_2$ is proportional to $\varpi$, the invariant horizontal 
differential and invariant vertical differential are independent of the choice 
of invariant contact forms that we use to complete the coframe. However, 
the expressions for the $\mathcal{A}, \mathcal{B}$ operators do depend on 
this choice. By the previous computation, we have $d_{\mathcal{V}}q$ in terms of the invariantized contact form $\iota(\theta_0)$:
\begin{equation}\label{eq:Aop_iota_Moeb2D}
    \begin{split}
        d_{\mathcal{V}}q & = \mathcal{A}(\iota(\theta_0)) = \left(4 D_{\sigma}^5 - 2 q D_{\sigma}^3 - \frac52 q_{\sigma} D_{\sigma}^2 + \left( \frac{q^2}{4} + 16 - q_{2\sigma}\right) D_{\sigma} + \left(\frac{qq_{\sigma}}{8} - \frac{q_{3\sigma}}{2}\right)  \right)\, \iota(\theta_0)
    \end{split}
\end{equation}
By application of formulas (\ref{eq:varthetatilde}), (\ref{eq:vartheta0}), 
we get a rational invariant contact form of order $0$ by setting
\begin{equation*}
        \vartheta_0  := \frac14\, \frac{dq(\partial_{y_5})}{\nabla(x)^5}  \,\theta_0.
\end{equation*}
Evaluating $\vartheta_0$ on the cross-section $\Sigma$ gives $\frac{\theta_0}{(4 
y_6)^5}$ and so we obtain that $\vartheta_0=\frac{\iota(\theta_0)}{q_{\sigma}^5}$. 
In expression (\ref{eq:Aop_iota_Moeb2D}) we now substitute 
$\iota(\theta_0) = q_{\sigma}^5 \vartheta_0 = (q_1)^{\frac52} \vartheta_0$ 
and $D_{\sigma}=\frac{\nabla}{\sqrt{q_1}}$. 
Upon repeatedly applying the product rule and collecting terms, 
we obtain $d_{\mathcal{V}}q$ expressed through a $\nabla$-differential 
operator $\mathcal{A}(\vartheta_0)$ acting on $\vartheta_0$. 
The formal adjoint is then readily computed as
\begin{equation*}
    \begin{split}
\mathcal{A}^* & = -4 \nabla^5 +\frac{30 q_2}{q_1} \nabla^4 +\left(2 q q_1-\frac{125 q_2^{2}}{q_1^{2}}+\frac{40 q_3}{q_1}\right) \nabla^{3} 
+\left(\frac{7 q_1^{2}}{2}-\frac{255 q_2 q_3}{q_1^{2}}+\frac{675 q_2^{3}}{2 q_1^{3}}-6 q q_2+\frac{30 q_4}{q_1}\right) \nabla^2 \\ &  + \left(-16 q_1^{2} +\frac{1403 q_2^{2} q_3}{2 q_1^{3}}-\frac{129 q_2 q_4}{q_1^{2}}-\frac{565 q_2^{4}}{q_1^{4}}+\frac{19 q q_2^{2}}{2 q_1}-\frac{88 q_3^{2}}{q_1^{2}}+\frac{12 q_5}{q_1}-\frac{q_1^{2} q^{2}}{4}-4 q q_3-\frac{15 q_2 q_1}{4}\right) \nabla +\alpha,
    \end{split}
\end{equation*}
where 
\begin{equation}\label{MoebAlpha}
    \begin{split}
\alpha= & \frac{2q_6}{q_1}
-\frac{26q_2q_5}{q_1^2}-\frac{45q_3q_4}{q_1^2}+\frac{359q_2^2q_4}{2q_1^3}-qq_4
+\frac{493q_2q_3^2}{2q_1^3}-\frac{797q_2^3q_3}{q_1^4}+\frac{13qq_2q_3}{2q_1}\\ 
& -\frac{3q_1q_3}{2}+\frac{455q_2^5}{q_1^5}- \frac{7qq_2^3}{q_1^2}
+\frac{9q_2^2}{4}+\frac{q^2q_1q_2}{8}+8q_1q_2-\frac{3qq_1^3}{8}.
    \end{split}
\end{equation}

Next, by application of the Leibniz rule we have that 
\begin{equation*}
    \begin{split}
d_{\mathcal{V}}\varpi_2 & = d_{\mathcal{V}}\left(\frac{\varpi}{q_{\sigma}}\right) 
= \frac{1}{q_{\sigma}}d_{\mathcal{V}}\varpi+ d_{\mathcal{V}}(q_{\sigma}^{-1}) \wedge \varpi \\
& = \frac{1}{q_{\sigma}}\mathcal{B}(\iota(\theta_0))\wedge\varpi 
-\frac{1}{q_{\sigma}^2}d_{\mathcal{V}}(q_{\sigma})\wedge\varpi 
= \left(\mathcal{B}(\iota(\theta_0))-\frac{d_{\mathcal{V}}(q_{\sigma})}{q_{\sigma}}  \right)\wedge\varpi_2,
    \end{split}
\end{equation*}
which implies that
 \begin{equation}\label{eqn_Moeb2D_Bequality}
\mathcal{B}(\vartheta_0) = \mathcal{B}(\iota(\theta_0))
-\frac{d_{\mathcal{V}}(q_{\sigma})}{q_{\sigma}}. 
 \end{equation}
By application of formula (\ref{eqn_inv_der_on_functions}) and using that 
$q_{\sigma} = \iota(4y_6)$, we compute that
\begin{equation*}
    \begin{split}
d_{\mathcal{V}}(q_{\sigma}) & = 4 \iota(d_Vy_6) + \sum_{k=1}^6 \iota(X_k(4 y_6))\, \epsilon^k  = 4 \iota(\theta_6) - \frac92 q\, \iota(\theta_4) - \frac52 q_{\sigma} \iota(\theta_3) - 200\, \iota(\theta_2) \\
& = \left[4 D_{\sigma}^6 -2 q D_{\sigma}^4 +\left(\frac{1}{2}-5 q_{\sigma}\right) D_{\sigma}^3  + \left(-\frac{9 q_{2\sigma}}{2}+16+\frac{q^{2}}{4}\right) D_{\sigma}^2 \right. \\ & \left. + \left(-\frac{1}{8} q+\frac{3}{4} q q_{\sigma}-\frac{5}{2} q_{3\sigma}\right) D_{\sigma} + \left(-\frac{1}{16} q_{\sigma}+\frac{3}{16} q_{\sigma}^{2}+\frac{1}{8} q q_{2\sigma}-\frac{1}{2} q_{4\sigma}\right) \right] \iota(\theta_0).
    \end{split}
\end{equation*}
We proceed as before, we substitute 
$\iota(\theta_0) = (q_1)^{\frac52} \vartheta_0$ and 
$D_{\sigma}  = \frac{\nabla}{\sqrt{q_1}}$ into (\ref{eqn_Moeb2D_Bequality}) 
and rewrite to obtain the expression for the $\nabla$-differential operator 
$\mathcal{B}(\vartheta_0)$. Then we compute its formal adjoint to be equal to
\begin{equation*}
    \begin{split}
\mathcal{B}^{*} & = -\frac{4}{q_1}\nabla^6 +\frac{54q_2}{q_1^2}\nabla^5 
+\left(\frac{100q_3}{q_1^2}+2q-\frac{395q_2^2}{q_1^3}\right)\nabla^4
+\left(-\frac{14q q_2}{q_1}-\frac{1195q_2 q_3}{q_1^3}+\frac{5q_1}{2}+\frac{110q_4}{q_1^2}
    +\frac{3835q_2^3}{2q_1^4}\right)\nabla^3 \\ &
+\left(-\frac{16q q_3}{q_1}+\frac{103q q_2^2}{2 q_1^2}-16 q_1-\frac{q_1 q^2}{4}-\frac{45 q_2}{4}
    -\frac{12635 q_2^4}{2 q_1^5}-\frac{688 q_3^2}{q_1^3}+\frac{13313 q_2^2 q_3}{2 q_1^4}
    -\frac{999 q_2 q_4}{q_1^3}+\frac{72 q_5}{q_1^2}\right)\nabla^2 \\ & 
+\left(-\frac{9 q q_4}{q_1}-\frac{110 q q_2^3}{q_1^3}+\frac{5q^2 q_2}{8}-\frac{q_1^2 q}{8}
    +\frac{26 q_2^2}{q_1}+\frac{13090 q_2^5}{q_1^6}+\frac{26 q_6}{q_1^2}-\frac{39 q_3}{4}
    +\frac{161 q q_2 q_3}{2 q_1^2} \right. \\ & \left. 
\hspace{25pt}  -\frac{39725 q_2^3 q_3}{2q_1^5}
    +\frac{7515 q_2^2 q_4}{2 q_1^4}+\frac{10415 q_2 q_3^2}{2q_1^4}-\frac{440 q_2 q_5}{q_1^3}
    -\frac{775 q_3 q_4}{q_1^3}+40 q_2\right) \nabla +\beta,
    \end{split}
\end{equation*}
where
\begin{equation}\label{MoebBeta}
    \begin{split}
\beta= & \frac{91 q_2 q_3}{4 q_1}+\frac{q^2 q_3}{4}+16 q_3-3 q_4+\frac{835 q_2^2 q_5}{q_1^4}
    -\frac{172 q_3 q_5}{q_1^3}-\frac{5 q^2 q_2^2}{8 q_1}+\frac{q_1 q q_2}{4}
    +\frac{23 q q_2 q_4}{q_1^2}\\ & 
    -\frac{132 q q_2^2 q_3}{q_1^3}-\frac{2 q q_5}{q_1}+\frac{16 q q_3^2}{q_1^2}
    -\frac{5675 q_2^3 q_4}{q_1^5}-\frac{11864 q_2^2 q_3^2}{q_1^5}-\frac{80 q_2 q_6}{q_1^3}+\frac{26180 q_2^4 q_3}{q_1^6}-\frac{40 q_2^2}{q_1} \\ 
& -\frac{109 q_2^3}{4 q_1^2} -\frac{13090 q_2^6}{q_1^7}
+\frac{4q_7}{q_1^2}
    +\frac{688 q_3^3}{q_1^4}-\frac{110 q_4^2}{q_1^3}+\frac{110 q q_2^4}{q_1^4}
    +\frac{q_1^3}{8}+\frac{2969 q_2 q_3 q_4}{q_1^4}.
    \end{split}
\end{equation}

The invariant Euler--Lagrange equations for an invariant Lagrangian 
$\lambda = L\,\Delta$ are then computed by the standard formula
$\mathcal{A}^{*}\mathcal{E}(L) - \mathcal{B}^{*}\mathcal{H}(L)= 0$. The matrix relative 
invariant $W$ determined by the global invariant forms $(\vartheta_0, \varpi_2)$ is 
proportional to (\ref{eq:W_Moebius}) with proportionality factor $q_{\sigma}^{-6}$.

For $L=1$, i.e.\ $\lambda=q_{\sigma}^{-1}\,\sigma$, 
the invariant Euler--Lagrange equations are given by $\beta=0$.
For $L=q$ the invariant Euler--Lagrange equations are given by 
 $\alpha=-\mathcal{B}^{*}(q)$. Here $\alpha$ and $\beta$ are given by formulas
\eqref{MoebAlpha} and \eqref{MoebBeta}, respectively.
For the Lagrangian $L=q_1$, i.e.\ 
$\lambda=q_1\Delta=d_{H}(q)$, the Euler--Lagrange equations 
are trivial, as expected for a divergence term.

\subsection{Projectively Invariant Variational Problems in $2D$}\label{subsecproj}

The projective linear group $PSL(3,\R)$ acts on the projective plane $\R P^2$. 
The action is given by the formula
 \begin{equation}
A\cdot[1:x:y]=[a_{00}+a_{01}x+a_{02}y:
a_{10}+a_{11}x+a_{12}y:a_{20}+a_{21}x+a_{22}y]. 
 \end{equation}
Thus, $A\in PSL(3)$ defines a linear fractional transformation 
$\tilde x = \frac{a_{10}+a_{11}x+a_{12}y}{a_{00}+a_{01} x+a_{02}y}$, 
$\tilde y = \frac{a_{20}+a_{21}x+a_{22}y}{a_{00}+a_{01} x+a_{02}y}$. 
The corresponding Lie algebra $\mathfrak{g}=\mathfrak{sl}(3,\R)$ 
in the open chart $\R^2(x,y)$ is generated by the vector fields 
$\p_x$, $\p_y$, $x\p_x$, $x\p_y$, $y\p_x$, $y\p_y$, $x(x\p_x+y\p_y)$, $y(x\p_x+y\p_y)$, which we enumerate as $X_1,\dots,X_8$.

\subsubsection{Generators of differential invariants}\label{StrLnsCon}

The local projective invariants are well-known, see \cite{Wilczynski06} and also
\cite{KL}. Here we correct an error in the formula for $\Theta_8$ in 
the latter reference. Define
 $$
\Theta_3=\frac{9y_2^2y_5-45y_2y_3y_4+40y_3^3}{y_2^3},\qquad
\Theta_8=6\Theta_3\frac{d^2\Theta_3}{dx^2}-7\Bigl(\frac{d\Theta_3}{dx}\Bigr)^2
+27\Theta_3^2R_2^{1/3}\frac{d^2}{dx^2}\Bigl(R_2^{-1/3}\Bigr).
 $$
Here $R_2=y_2$ is a relative differential invariant, and we have two other polynomial relative invariants
$R_5=R_2^3\Theta_3$ and $R_7=R_2^8\Theta_8$. 
These relative invariants have weights $(2,1)$, $(9,3)$ and $(24,8)$,
respectively, with respect to the basis of the Cartan subalgebra of 
$\mathfrak{sl}(3,\R)$ corresponding to vector
fields $(-x\p_x,y\p_y)$. This implies that the relative invariants $\frac{R_5^3}{R_2R_7}$ and $\frac{R_2^3R_7^2}{R_5^6}$ have
weights, respectively, $(1,0)$ and $(0,1)$ and thus generate the weight lattice. 
Above we use the convention that indices of $R_i$ correspond to the order 
of the invariant, while indices of $\Theta_j$ corresponds to the second weight 
(the space of weights of relative differential invariants for $\mathfrak{g}$ is two-dimensional, see \cite{KS2}).

The projective differential invariants are generated by one differential invariant (projective curvature) and one invariant derivation (projective parameter):
 \begin{gather}
\kappa= \frac{4\Theta_8}{9 \sqrt[3]{\Theta_3}^8},\qquad \nabla=\frac{1}{ \sqrt[3]{\Theta_3}}  \frac{d}{dx}.
 \end{gather}
Thus we get one new differential invariant in every order starting from 7: $\kappa_i := \nabla^i(\kappa), i \geq 1$. Furthermore, choosing 
\[\varpi=\sqrt[3]{\Theta_3} \left(dx- \left(\frac{D_x(\Theta_3)}{3R_2\Theta_3}+\frac{2}{3} D_x(1/R_2)\right)\theta_0 \right) = \iota(3^{2/3} dx)\] ensures that $\varpi(\nabla)=1$.  

Notice that the differential equation $R_2=0$ describes the straight lines, while $R_5=0$ describes conics \cite{Halphen78}. For an overview of some of the distinguished projectively invariant families of curves, we refer to \cite{Konovenko}.   

\subsubsection{Projectively-invariant Euler--Lagrange equations via local invariants} 

In principle, we could compute the Euler--Lagrange operator by computing $d_{\mathcal{V}}\kappa$ and $d_{\mathcal{V}} \varpi$ and writing them in terms of ${\nabla}^k(\vartheta_0)$ and ${\nabla}^k(\kappa)$ by eliminating jet-variables. However, due to the size of the formulas, this is not practical even with the use of computer algebra systems. Therefore, we will instead do these computations on a cross-section, with an approach that differs slightly from the algorithm in Section \ref{sect:movingframe}.

We use the transversal cross-section defined by 
 \begin{equation}\label{Proj-cross}
\Sigma=\{x=0,\ y_0=0,\ y_1=0,\ y_2=1,\ y_3=0,\ y_4=0,\ y_5=1,\ y_6=0\}. 
 \end{equation}
Restricting the projective curvature $\kappa$ to the cross-section gives 
$\frac{8}{3^{7/3}}\,y_7$, in other words, $\kappa=\iota\left(\frac{8}{3^{7/3}}\, y_7 \right)$. 

The cross-section gives the invariantizations of the jet coordinates up to 
order 6 and using $\iota(y_7) = \frac{3^{7/3}\kappa}{8}$, we can 
compute all the $\gamma$  terms. Then using the formulas for 
invariant differentials, we compute all the other ingredients. We obtain the $\gamma$-terms as explained in step 1 and 2 of the algorithm in Section \ref{sect:movingframe}:
\begin{gather*}
    \gamma^1 = -\iota(dx),\qquad \gamma^2 = 0, \qquad \gamma^3 = 0, \qquad\gamma^4 = -\iota(dx), \\\gamma^5 = -\frac{3^{4/3} \kappa\, \iota(dx)}{8}, \qquad\gamma^6 = 0,\qquad \gamma^7 = \frac{3^{4/3}\kappa\, \iota(dx)}{8},\qquad \gamma^8 = \frac{ \iota(dx)}{6}.
\end{gather*}
Here we will not follow the algorithm completely. Instead of using the invariant contact forms $\iota(\theta_i)$ as basis elements, we will use $\nabla^i(\vartheta_0)$. This lets us skip the computation of the $\epsilon$-terms. For $\vartheta_0$ we choose 
 \[ 
\vartheta_0 = \frac{\Theta_8}{I_0 R_2 \Theta_3^2} \theta_0 = \iota\left(\frac{3^{10/3}}{4} \theta_0\right). 
 \] 

Consider a general invariant contact form $\vartheta$, and let $\theta$ be a contact form satisfying $\vartheta = \iota(\theta)$. Note that there is some freedom in choosing $\theta$, the only requirements being that $\theta$ is a contact form and $\vartheta|_{\Sigma} = \theta|_{\Sigma}$. We will choose $\theta$ as follows. Let $\varrho \colon J^\infty \to \Sigma$ be the projection defined by 
\[ \varrho \colon (x,y_0,y_1,y_2,y_3,y_4,y_5,y_6,y_7,\dots, y_k,\dots) \mapsto (0,0,0,1,0,0,1,0,y_7,\dots, y_k, \dots),\]
 and define $\theta = \sum_{i \geq 0} \varrho^*(\vartheta(\partial_{y_i})|_\Sigma) \theta_i$, which has the benefit of a much simpler algebraic expression than $\vartheta$. 
 
 Using  \eqref{def_inv_hor_der} together with the identities $d_H(\theta) = dx \wedge D_x(\theta)$ and $d_{\mathcal{H}}(\vartheta) = \varpi \wedge \nabla (\vartheta)$ gives 
\[\varpi \wedge \nabla(\vartheta) = \iota(dx) \wedge \iota (D_x(\theta)) + \sum_k \gamma^k \wedge \iota(X_k(\theta)),  \]
or, if we evaluate on the cross-section: 
\[ \left(\varpi \wedge \nabla(\vartheta) \right)|_{\Sigma} =\left( dx \wedge D_x(\theta) + \sum_k \gamma^k \wedge X_k(\theta)\right)\Bigg|_{\Sigma}.\]
Contracting the 2-form with $\nabla$ (at the points of ${\Sigma}$) results in 
\begin{equation} \label{eq:OnCrossection}
    \nabla(\vartheta)|_{\Sigma} = \nabla(\theta)|_{\Sigma} + (\nabla|_{\Sigma}) \lrcorner \left( \sum_k \gamma^k \wedge X_k(\theta) \right)\Bigg|_{\Sigma}
\end{equation}
and we thus get the expression for $\nabla(\vartheta)|_{\Sigma}$ from the expression $\vartheta|_{\Sigma}$. This lets us bypass the brute-force computation of $\nabla(\vartheta)$, which becomes impractical when we need to compute $\nabla^i(\vartheta_0)$ for $i$ up to 7. Instead, we use \eqref{eq:OnCrossection} to compute $\nabla^i(\vartheta_0)|_{\Sigma}$.  This is sufficient since the invariant objects are completely determined by their behaviour on the cross-section. 

Next, we compute $d_{\mathcal{V}} \kappa |_{\Sigma}$ and $ d_{\mathcal{V}}\varpi|_{\Sigma}$ and rewrite them in terms of $\nabla^i(\vartheta_0)|_\Sigma$. After substituting the remaining free jet variables $y_7,y_8, \dots$,  for the appropriate differential invariants $\kappa|_\Sigma, \kappa_1|_{\Sigma}, \dots$, we get the following expressions:

\begin{equation}
    \begin{split}
        d_{\mathcal{V}}\kappa &=\Bigg[\frac{32}{3} \nabla^7-16 \kappa \nabla^5 -\frac{140}{3} \kappa_1 \nabla^4-\left(\frac{196}{3} \kappa_2-6 \kappa^2\right) \nabla^3-\frac{1}{3} (140 \kappa_3-55 \kappa \kappa_1) \nabla^2 \\&-\frac{1}{3} \left(52 \kappa_4-41 \kappa \kappa_2-32 \kappa_1^2-\frac{2}{3} \kappa_1+2\kappa^3-\frac{8}{27}\right) \nabla -\frac{1}{3} ( 8 \kappa_5-10 \kappa \kappa_3-2(12\kappa_2-\kappa^2)\kappa_1)\Bigg]\left(\vartheta_0 \right)  \\
        d_{\mathcal{V}}\varpi 
        & = \left[\frac{4}{3} \nabla^5-\frac{5}{3} \kappa \nabla^3-\left(\frac{5}{2} \kappa_1-\frac{2}{9}\right) \nabla^2 -\left( \frac{3}{2} \kappa_2-\frac{1}{3}\kappa^2\right)\nabla -\frac{1}{3} (\kappa_3-\kappa \kappa_1) \right]\left(\vartheta_0\right) \wedge \varpi.
    \end{split}
\end{equation}
By using the definition of the formal adjoint, we compute that 
\begin{equation}
    \begin{split}
\mathcal{A}^{*} & = \Bigg[ -\frac{32}{3} \nabla^7+16 \kappa \nabla^5 +\frac{100}{3} \kappa_1 \nabla^4+\left(\frac{116}{3} \kappa_2-6 \kappa^2 \right)\nabla^3 +\left(\frac{88}{3} \kappa_3-\frac{53}{3} \kappa \kappa_1\right) \nabla^2 \\
&+ \left(\frac{40}{3} \kappa_4-13 \kappa \kappa_2-10 \kappa_1^2-\frac{2}{9} \kappa_1+\frac{2}{3} \kappa^3-\frac{8}{81}\right) \nabla 
+\left(\frac{8}{3} \kappa_5-4 \kappa \kappa_3-8 \kappa_1 \kappa_2-\frac{2}{9} \kappa_2+\frac{4}{3} \kappa^2 \kappa_1 \right)\Bigg]
    \end{split}
\end{equation}
and
\begin{equation}
    \begin{split}
        \mathcal{B}^{*} & = \left[-\frac{4}{3}\nabla^5+\frac{5}{3} \kappa \nabla^3+\left(\frac{5}{2} \kappa_1+\frac{2}{9}\right) \nabla^2+\left(\frac{3}{2} \kappa_2 - \frac{1}{3} \kappa^2\right)\nabla +\frac{1}{3} (\kappa_3-\kappa \kappa_1)\right].   
    \end{split}
\end{equation}
From the formula $\vartheta_0 \wedge \varpi = W \theta_0 \wedge dx$, 
we compute 
\begin{equation}\label{eq:W_Proj2D}
    W =  \frac{9 \Theta_3}{4 y_2}.
\end{equation}

The following table shows the Euler--Lagrange equations corresponding to a few simple Lagrangians. 

\medskip

\makegapedcells
\begin{tabular}{||c c||} 
 \hline
 Lagrangian  & Invariant Euler--Lagrange Equations \\ [0.5ex] 
 \hline\hline
$\varpi$ & $ \kappa_3 =  \kappa \kappa_1 $ \\ 
 \hline
$\kappa \varpi$ &  $ \kappa_5 = \frac{3}{2} \kappa \kappa_3+3 \kappa_1\kappa_2-\frac{1}{2} \kappa^2 \kappa_1  $\\
 \hline
$\kappa^2 \varpi$ &  $ \kappa_7 =\frac{13}{8} \kappa \kappa_5+\frac{15}{4} \kappa_1 \kappa_4+\frac{41}{8} \kappa_2 \kappa_3-\frac{49}{64} \kappa^2 \kappa_3 -\frac{89}{32} \kappa \kappa_1 \kappa_2-\frac{45}{64} \kappa_1^3+\frac{9}{64} \kappa^3 \kappa_1-\frac{1}{108} \kappa_1  $ \\
 \hline
$\kappa_1^2 \varpi$  & 
 \hspace{-40pt}
$ \kappa_9 =  \frac{3}{2} \kappa \kappa_7+ 3 \kappa_1 \kappa_6  + \left(\frac{13}{4} \kappa_2- \frac{9}{16} \kappa^2\right)  \kappa_5 + \left(\frac{11}{4} \kappa_3- \frac{3}{2}\kappa \kappa_1\right)\kappa_4$ \\ 
 & \hspace{20pt}
$ +\left(\frac{1}{16}\kappa^3- \frac{11}{16}\kappa_1^2 -\frac{9}{8}\kappa \kappa_2- \frac{1}{108} \right) \kappa_3 -\frac{3}{8} \kappa_1 \kappa_2^2  + \frac{3}{32}  \kappa^2 \kappa_1 \kappa_2   - \frac{1}{64} \kappa \kappa_1^3 $\\
 \hline
\end{tabular}

\medskip

In these equations, we have omitted the overall (relative invariant) factor $W$. As a consequence, the formulas above only cover the regular extremals. The curves defined by $\Theta_3=0$ are singular extremals (conics and straight lines in
the limit, cf.\ \S\ref{StrLnsCon}) where $\varpi$ vanishes or is infinite. 
The Euler--Lagrange equation corresponding to $\varpi$ was also computed in \cite{Konovenko}. The differences in our formula compared to theirs is explained by the fact that we have chosen a different set of generators for the differential invariants and that we have omitted the factor $W$ from the equation.

\subsubsection{Projectively invariant Euler--Lagrange equations via global invariants} Let us now compute the Euler--Lagrange operator in terms of rational differential invariants. 

The field of rational differential invariants is generated by 
\[K_0 =\frac{\Theta_8^3}{3\Theta_3^8}, \qquad \bar{\nabla}= \frac{3 \Theta_3^5}{\Theta_8^2} D_x. \] Notice that $K_0$ is proportional to $\kappa^3$. In order to express the Euler--Lagrange operator in terms of these, we use the global invariant 1-forms 
 \[ 
\varpi = \frac{\Theta_8^2}{3 \Theta_3^5} \left( dx - \left(\frac{D_x(\Theta_3)}{3 R_2 \Theta_3}+\frac{2D_x(1/R_2) }{3} \right) \theta_0\right), \qquad \vartheta_0 = \frac{3 \Theta_3^6}{R_2 \Theta_8^2} \theta_0, 
 \]
which gives the scaling factor $W= \frac{\Theta_3}{3 y_2}$. 
We use a similar approach as above, with the same cross-section given by \eqref{Proj-cross}. 
The computation is straight-forward 
and we obtain the expressions

\begin{align*}
\mathcal{A}^*  &= 
-2K_0^6\bar{\nabla}^7-\tfrac{140}3 K_0^5K_1\bar{\nabla}^6
+\left(-\tfrac{308}3 K_0^5K_2 +4K_0^5 -\tfrac{3164}9 K_0^4K_1^2\right) \bar{\nabla}^5 
+\big(-140K_0^5K_3 -\tfrac{11060}9K_0^4K_1K_2 \\ & 
 +\tfrac{505}{9}K_0^4K_1 -\tfrac{27440}{27}K_0^3K_1^3\big) \bar{\nabla}^4 
+ \big(-\tfrac{364}{3}K_0^5K_4 -\tfrac{3920}{3}K_0^4K_1K_3 
-\tfrac{7420}{9}K_0^4K_2^2 +\tfrac{749}{9}K_0^4K_2 -2K_0^4 \\ &
-\tfrac{35840}{9}K_0^3K_1^2K_2 +\tfrac{5980}{27}K_0^3K_1^2
-\tfrac{84560}{81}K_1^4K_0^2 \big) \bar{\nabla}^3 
+\big( -\tfrac{196}{3}K_0^5K_5 -\tfrac{7532}{9}K_0^4K_1K_4
-\tfrac{11480}{9}K_0^4K_2K_3 \\ &
+\tfrac{622}{9}K_0^4K_3 -\tfrac{82880}{27}K_0^3K_1^2K_3
-\tfrac{33040}{9}K_0^3K_1K_2^2 +\tfrac{12526}{27}K_0^3K_1K_2
-\tfrac{377}{27}K_0^3K_1 -\tfrac{297920}{81}K_1^3K_0^2K_2 \\ &
+\tfrac{21500}{81}K_1^3K_0^2 -\tfrac{66080}{243}K_0K_1^5 \big) \bar{\nabla}^2 
+ \big( -20K_0^5K_6-\tfrac{2680}{9}K_0^4K_1K_5 -\tfrac{4796}{9}K_0^4K_2K_4
-320K_0^4K_3^2 \\ &
+\tfrac{274}{9}K_0^4K_4 -\tfrac{34768}{27}K_0^3K_1^2K_4 
-\tfrac{97520}{27}K_0^3K_1K_2K_3 -\tfrac{18640}{27}K_0^3K_2^3
+\tfrac{6730}{27}K_0^3K_1K_3 +\tfrac{1382}{9}K_0^3K_2^2 \\ &
-\tfrac{97}{9}K_0^3K_2 +\tfrac{8}{27}K_0^3 -\tfrac{146240}{81}K_0^2K_1^3K_3
-\tfrac{79160}{27}K_0^2K_1^2K_2^2 +\tfrac{40828}{81}K_0^2K_1^2K_2
-\tfrac{1676}{81}K_0^2K_1^2 -\tfrac{1}{6}K_0^2K_1 \\ &
-\tfrac{3}{2}K_0^2 -\tfrac{160160}{243}K_0K_1^4K_2 +\tfrac{16880}{243}K_0K_1^4
-\tfrac{1120}{243}K_1^6\big) \bar{\nabla} + \alpha,
\end{align*}
where
 \begin{align*}
\alpha &= 
-\tfrac{8}{3}K_0^5K_7 -\tfrac{136}{3}K_0^4K_1K_6 -\tfrac{280}{3}K_0^4K_2K_5
-\tfrac{6128}{27}K_0^3K_1^2K_5 +\tfrac{50}{9}K_0^4K_5 
-\tfrac{392}{3}K_0^4K_3K_4 -\tfrac{19856}{27}K_0^3K_1K_2K_4 \\ &
-\tfrac{3344}{9}K_0^2K_1^3K_4 +\tfrac{1460}{27}K_0^3K_1K_4
-\tfrac{1280}{3} K_0^3K_1K_3^2 -\tfrac{12880}{27}K_0^3K_2^2K_3
-\tfrac{36320}{27}K_0^2K_1^2K_2K_3 +\tfrac{2144}{27}K_0^3K_2K_3 \\ &
-\tfrac{37280}{243}K_0K_1^4K_3 +\tfrac{10504}{81}K_0^2K_1^2K_3
-\tfrac{28}{9}K_0^3K_3 -\tfrac{35920}{81}K_0^2K_1K_2^3
-\tfrac{64000}{243}K_0K_1^3K_2^2 +\tfrac{436}{3}K_0^2K_1K_2^2 
-\tfrac{1}{6}K_0^2K_2 \\ &
+\tfrac{17584}{243}K_0K_1^3K_2 -\tfrac{1120}{243}K_1^5K_2
-\tfrac{1208}{81}K_0^2K_1K_2 +\tfrac{200}{243}K_1^5 -\tfrac{44}{9}K_0K_1^3 
-\tfrac{2}{9}K_0K_1^2 +\tfrac{16}{27}K_0^2K_1 -2K_0K_1,
 \end{align*}
and
 \begin{align*}
\mathcal{B}^* &= 
-\tfrac{4}{3}K_0^5\bar{\nabla}^7 -\tfrac{196}{9}K_0^4K_1\bar{\nabla}^6 
+\big( -\tfrac{364}{9}K_0^4K_2 +\tfrac{23}{9}K_0^4 
-\tfrac{2800}{27}K_0^3K_1^2\big) \bar{\nabla}^5 
+\big(-\tfrac{140}{3}K_0^4K_3 -\tfrac{7840}{27}K_0^3K_1K_2 \\ &
+\tfrac{620}{27}K_0^3K_1 -\tfrac{12880}{81}K_1^3K_0^2\big) \bar{\nabla}^4 
+ \big(-\tfrac{308}{9}K_0^4K_4 -\tfrac{2240}{9}K_0^3K_1K_3
-\tfrac{3920}{27}K_0^3K_2^2 +\tfrac{82}{3}K_0^3K_2 -\tfrac{31}{27}K_0^3 \\ &
-\tfrac{12320}{27}K_0^2K_1^2K_2 +\tfrac{3940}{81}K_0^2K_1^2
-\tfrac{14560}{243}K_0K_1^4\big) \bar{\nabla}^3 
+ \big( -\tfrac{140}{9}K_0^4K_5 -\tfrac{3472}{27}K_0^3K_1K_4
-\tfrac{4480}{27}K_0^3K_2K_3 \\ &
+\tfrac{494}{27}K_0^3K_3 -\tfrac{21280}{81}K_0^2K_1^2K_3
-\tfrac{7280}{27}K_0^2K_1K_2^2 +\tfrac{5960}{81}K_0^2K_1K_2
-\tfrac{725}{162}K_0^2K_1 +\tfrac{1}{6}K_0^2 -\tfrac{31360}{243}K_0K_1^3K_2 \\ &
+\tfrac{5200}{243}K_0 K_1^3 -\tfrac{1120}{729}K_1^5\big) \bar{\nabla}^2 
+ \big(-4K_0^4K_6 -\tfrac{992}{27}K_0^3K_1K_5 -\tfrac{1360}{27}K_0^3K_2K_4
-\tfrac{80}{3}K_0^3K_3^2 +\tfrac{176}{27}K_0^3K_4 +\tfrac{40}{81}K_1^4 \\ &
-\tfrac{6704}{81}K_0^2K_1^2K_4 -\tfrac{13600}{81}K_0^2K_1K_2K_3
-\tfrac{2000}{81}K_0^2K_2^3 +\tfrac{2336}{81}K_0^2K_1K_3 
+\tfrac{1190}{81}K_0^2K_2^2 -\tfrac{427}{162}K_0^2K_2 +\tfrac{4}{27}K_0^2 \\ &
-\tfrac{10240}{243}K_0K_1^3K_3 -\tfrac{3760}{81}K_0K_1^2K_2^2 
+\tfrac{208}{9}K_0K_1^2K_2 -\tfrac{209}{81}K_0K_1^2 +\tfrac{2}{9}K_0K_1
-K_0 -\tfrac{1120}{729}K_1^4K_2  \big) \bar{\nabla} +\beta,
 \end{align*}
where
 \begin{align*}
\beta &= 
-\tfrac{4}{9}K_0^4K_7 -\tfrac{40}{9}K_0^3K_1K_6 -\tfrac{56}{9}K_0^3K_2K_5
-\tfrac{880}{81}K_0^2K_1^2K_5 +\tfrac{26}{27}K_0^3K_5 \\ &
-\tfrac{56}{9}K_0^3K_3K_4 -\tfrac{1696}{81}K_0^2K_1K_2K_4 
-\tfrac{160}{27}K_0K_1^3K_4 +\tfrac{124}{27}K_0^2K_1K_4 
-\tfrac{80}{9}K_0^2K_1K_3^2 -\tfrac{560}{81}K_0^2K_2^2K_3 \\ &
-\tfrac{800}{81}K_0K_1^2K_2K_3 +\tfrac{38}{9}K_0^2K_2K_3 -\tfrac{160}{729}K_1^4K_3 
+\tfrac{880}{243}K_0K_1^2K_3 -\tfrac{49}{81}K_0^2K_3 \\ &
-\tfrac{320}{243}K_0K_1K_2^3 +\tfrac{544}{243}K_0K_1K_2^2 -\tfrac{80}{729}K_1^3K_2^2
+\tfrac{40}{243}K_1^3K_2 -\tfrac{92}{81}K_0K_1K_2 -\tfrac{5}{81}K_1^3 
+\tfrac{4}{27}K_0K_1 -\tfrac{1}{3}K_1.
 \end{align*}

We notice that the formulas are more complex than with the previous local invariants. The point here is to show that the computations can be done with respect to any generating set of differential invariants, and any basis of invariant contact forms, without much additional effort. 

As an example of a Lagrangian, consider the global Lagrangian $\lambda = \varpi$.
This corresponds to $L=1$, in which case the Euler--Lagrange equations are
given by $\beta=0$. More explicitly, the regular extremals must satisfy
the following equation: 
\begin{align*}
    K_7 &= -\frac{1}{324 K_0^4} \big(3240 K_0^3 K_1 K_6+4536 K_0^3 K_2 K_5+4536 K_0^3 K_3 K_4+7920 K_0^2 K_1^2 K_5+15264 K_0^2 K_1 K_2 K_4\\&+6480 K_0^2 K_1 K_3^2+5040 K_0^2 K_2^2 K_3+4320 K_0 K_1^3 K_4+7200 K_0 K_1^2 K_2 K_3+960 K_0 K_1 K_2^3+160 K_1^4 K_3\\&+80 K_1^3 K_2^2-702 K_0^3 K_5-3348 K_0^2 K_1 K_4-3078 K_0^2 K_2 K_3-2640 K_0 K_1^2 K_3-1632 K_0 K_1 K_2^2-120 K_1^3 K_2\\&+441 K_0^2 K_3+828 K_0 K_1 K_2+45 K_1^3-108 K_0 K_1+243 K_1 \big).
\end{align*}

Similarly, for $L=q$ the invariant Euler--Lagrange equations are given by 
 $\alpha=-\mathcal{B}^{*}(q)$.

\subsection{Euclidean Invariant Variational Problems in $3D$}

The manifold is $\R^3(x,y,z)$ and we consider the Lie group action of $\text{SE}(3)$ on $
\mathbb{R}^3$. The corresponding Lie algebra $\mathfrak{se}(3)$ has generators $\partial_x, \partial_y, \partial_z,x \partial_y - y \partial_x,  x \partial_z - z \partial_x,  y \partial_z - z \partial_y$, which we denote by $X_1, \dots, X_6$, respectively.

\subsubsection{Generating differential invariants}
In this case the algebra of local differential invariants is generated by
 $$
\kappa=\frac{\sqrt{(y_1z_2-z_1y_2)^2+y_2^2+z_2^2}}{(1+y_1^2+z_1^2)^{3/2}},\quad
\tau=\frac{y_2z_3-z_2y_3}{(y_1z_2-z_1y_2)^2+y_2^2+z_2^2} \text{ and }\ 
D_s =\frac1{\sqrt{1+y_1^2+z_1^2}}D_x.
 $$
The field of rational differential invariants are generated by
 $$
\kappa^2,\quad \tau,\quad \nabla=\kappa\kappa_s D_s.
 $$

\subsubsection{Euclidean Invariant Euler--Lagrange via local invariants}

Recall that the classical Frenet--Serret moving frame for a nondegenerate curve is 
constructed from considering the unit tangent vector $T$ and normal $N$ and adding to 
these the cross-product $B = T \times N$. We shall view this moving frame as an 
equivariant map $\rho_{\text{FS}}: J^\infty \rightarrow \text{SO}(3)$ given by composition of $\pi^{\infty}_2$ with
 \begin{gather}\label{3D_def_Frenet_Serret_frame}
(x,y,z,y_1,z_1,y_2,z_2) \mapsto \begin{pmatrix}
\frac{1}{\ell} & - \frac{y_1y_2+z_1z_2}{\ell\zeta}  & \frac{y_1z_2-y_2z_1}{\zeta} \\
\frac{y_1}{\ell} & \frac{-y_1z_1z_2+y_2z_1^2+y_2}{\ell\zeta} & -\frac{z_2}{\zeta}\\
\frac{z_1}{\ell} & \frac{y_1^2z_2-y_1y_2z_1+z_2}{\ell\zeta} & \frac{y_2}{\zeta}\end{pmatrix}\\
\text{ where }\ell=\sqrt{1+y_1^2+z_1^2}\text{ and }\zeta=\sqrt{(y_1z_2-z_1y_2)^2+y_2^2+z_2^2}=\ell^3\kappa.
\label{ellzeta}
 \end{gather}

Recall that a moving frame assigns to a point $u \in J^{\infty}$ the group 
element in $\text{SO}(3)$ which transforms $u$ to the unique point lying in 
the intersection of the orbit of $u$ and the cross-section. 
To be more precise, we should actually consider 
the orbits of $\text{SE}(3)$ and not only those of $\text{SO}(3)$, but we may use the translations to set $x = 0,y = 0, z = 0$. The remaining equations for cross-section $\Sigma$ can be recovered by solving for the points $u \in J^2$ such that $\rho_{\text{FS}}(u) = \mathbf{1}_3$. The cross-section $\Sigma_{FS}$ for the Frenet--Serret moving frame is given by
\begin{equation}
    \Sigma_{\text{FS}} = \{x = 0, y = 0, z = 0, y_1 = 0, z_1 = 0, z_2 = 0 \}.
\end{equation}

Invariantizing the horizontal one-form $dx$ gives the invariant horizontal form $\varpi$:
\begin{equation}\label{3D_def_varpi}
    \varpi = \iota(dx) = \frac{dx + y_1 dy + z_1 dz}{\sqrt{1 + y_1^2 + z_1^2}}. 
\end{equation}
Similarly, we obtain invariant contact forms $\iota(\theta^i_j)$ by invariantizing the 
standard contact one-forms.

We use the algorithm to compute the invariant EL equations. By application of the 
invariant horizontal- and vertical differentials 
(\ref{eqn_inv_der_on_functions}) to the coordinate functions $x, y, z, y_1, z_1, z_2$ we 
get a linear system on the $\gamma$ and $\epsilon$ terms. The solution is given by
\begin{equation*}\label{3D_eqn_gammas}
    \begin{split}
        \gamma^1 &  = - \varpi,\  
        \gamma^2  = 0,\ 
        \gamma^3  = 0,\ 
        \gamma^4  = - \kappa \varpi,\  
        \gamma^5  = 0,\ 
        \gamma^6  = -\tau\varpi, \\
        \epsilon^1 & = 0,\
        \epsilon^2  = -\iota(\theta_0^y), \
        \epsilon^3  = - \iota(\theta_0^z), \
        \epsilon^4 = - \iota(\theta_1^y),\  
        \epsilon^5  = - \iota(\theta_1^z),\ 
        \epsilon^6  = -\frac{\iota(\theta_2^z)}{\kappa}.
    \end{split}
\end{equation*}

We compute the invariant vertical differential of the 
invariants $\kappa, \tau$ and the invariant horizontal form $
\varpi$, and then rewrite these expressions in terms of 
differential operators acting on the lowest order invariant 
contact forms. 

\medskip

Using that $\iota(y_2) = \kappa$, we compute
\begin{equation}\label{3D_inv_ver_der_kappa}
    \begin{split}
        d_{\mathcal{V}}\kappa &  =  \iota(\theta_2^y)  = \left(D_s^2 + (\kappa^2 - \tau^2) \right) \iota(\theta_0^y) - \left(2 \tau D_s + \tau_s \right) \iota(\theta_0^z).
    \end{split}
\end{equation}
Similarly, by using that $\iota\left(\frac{z_3}{y_2}\right) = \tau$ and equation (\ref{eqn_inv_der_on_functions}), we obtain that
\begin{equation}\label{3D_inv_ver_der_tau}
    \begin{split}
        d_{\mathcal{V}}\tau & =  \left(\frac{2 \tau}{\kappa} D_s^2 + \left(\frac{3 \tau_1}{\kappa} - \frac{2 \tau \kappa_s}{\kappa^2}\right) D_s + \left(\frac{\tau_{ss}}{\kappa} + 2 \tau \kappa - \frac{\tau_s \kappa_s}{\kappa^2}\right) \right) \iota(\theta_0^y) \\ & + \left(\frac{1}{\kappa} D_s^3 - \frac{\kappa_s}{\kappa^2} D_s^2 + \left(\kappa - \frac{\tau^2}{\kappa} \right) D_s + \left(\frac{\kappa_s \tau^2}{\kappa^2} - \frac{2 \tau \tau_s}{\kappa}\right) \right) \iota(\theta_0^z).
    \end{split}
\end{equation}

We have $d_{\mathcal{V}}\varpi = -\kappa \, \iota(\theta_0^y) \wedge \varpi$ from which we obtain that $\mathcal{B} = \begin{pmatrix} -\kappa & 0 \end{pmatrix}$.
By using the definition of the formal adjoint, we obtain that
\begin{equation}
    \mathcal{A}^{*} = \begin{pmatrix} (\mathcal{A}^{\kappa}_y)^{*} & (\mathcal{A}^{\tau}_y)^{*} \\
    (\mathcal{A}^{\kappa}_z)^{*} & (\mathcal{A}^{\tau}_z)^{*} \end{pmatrix} = 
\begin{pmatrix}
    D_s^2 + (\kappa^2 - \tau^2) &  \frac{2 \tau}{\kappa} D_s^2 + \left( \frac{\tau_s}{\kappa} - \frac{2 \tau \kappa_s}{\kappa^2} \right) D_s + 2 \tau \kappa   \\
    2 \tau D_s + \tau_s  & -\frac{1}{\kappa} D_s^3 + \frac{2 \kappa_s}{\kappa^2} D_s^2 + \left(\frac{\kappa_{ss}}{\kappa^2} - \frac{2 \kappa_s^2}{\kappa^3} - \kappa + \frac{\tau^2}{\kappa}   \right) D_s - \kappa_s.       
\end{pmatrix}
\end{equation}
Note that the computation of the formal adjoint $\mathcal{A}^{*}$ 
involves in particular the transposition of the matrix $
\mathcal{A}$. Clearly, we have that $\mathcal{B}$ is (formally) 
self-adjoint, that is, $\mathcal{B}^{*} = \mathcal{B}^T$.

\medskip

In component form, the invariant Euler--Lagrange equations 
are given by 
\begin{equation}\label{3D_inv_EulerLagrange}
    \begin{cases}
      \left(D_s^2 + (\kappa^2 - \tau^2) \right) \mathcal{E}_{\kappa}(L) + \left(\frac{2 \tau}{\kappa} D_s^2 + \left( \frac{\tau_s}{\kappa} - \frac{2 \tau \kappa_s}{\kappa^2} \right) D_s + 2 \tau \kappa \right) \mathcal{E}_{\tau}(L) + \kappa \cdot \mathcal{H}(L) = 0\\
      \left( 2 \tau D_s + \tau_s \right) \mathcal{E}_{\kappa}(L) + \left( -\frac{1}{\kappa} D_s^3 + \frac{2 \kappa_s}{\kappa^2} D_s^2 + \left(\frac{\kappa_{ss}}{\kappa^2} - \frac{2 \kappa_s^2}{\kappa^3} - \kappa + \frac{\tau^2}{\kappa}   \right) D_s - \kappa_s \right) \mathcal{E}_{\tau}(L)  = 0
    \end{cases}       
\end{equation}

The matrix relative invariant $W$ is readily computed, either by using 
the explicit moving frame $\rho_{\text{FS}}$ or by applying formulas 
(\ref{eq:varthetatilde}), (\ref{eq:vartheta0}), and using notations \eqref{ellzeta} 
we get the formula
\begin{equation}\label{eq:W_Eucl3D}
    W = \frac1\zeta \begin{pmatrix}
(1+z_1^2)y_2-y_1z_1z_2 & (1+y_1^2)z_2-y_1z_1y_2 \\
-z_2\ell &  y_2\ell \end{pmatrix}.
\end{equation}
In particular, we have $\det(W) = \ell$. Consequently, singular extremals 
must have null length. Such extremals do not exist for Euclidean signature 
(but they do appear for Lorentzian signature). 

Some examples of invariant Euler--Lagrange equations are displayed in the following table.

\medskip

\begin{tabular}{||c c c c||} 
 \hline
 Lagrangian & Eulerian & Hamiltonian & Invariant Euler--Lagrange Equations \\ [0.5ex] 
 \hline\hline
 $ds$ & $ (0,0)^T $ & $-1$ & $ \kappa = 0 $ \\ 
 \hline
 $\kappa\, ds$ & $(1,0)^T$ & $ -\kappa $ & $ \tau = 0  $\\
 \hline
 $\tau\, ds$ & $(0,1)^T  $ & $ -\tau $ & $\begin{cases}
      \tau\, \kappa = 0  \\
     \kappa_s = 0
    \end{cases} $ \\
 \hline
  $\kappa^2 \, ds$ & $(2\kappa,0)^T  $ & $ -\kappa^2 $ & $ \begin{cases}
     \kappa_{ss} = \kappa \tau^2 - \frac12 \kappa^3 \\
     \kappa \tau_s + 2 \kappa_s \tau = 0
    \end{cases} $ \\
\hline
\end{tabular}
\medskip

Notice that we only get 1 equation in the first two examples. For them the invariant Euler--Lagrange operator is given by $(-\kappa,0)$ and $(-L_0^2,L_1)$, respectively. 

\subsubsection*{Geodesics}

The variational problem $\int ds$ gives the invariant Euler--Lagrange equation 
$\kappa = 0$. Solution curves of this equation are straight lines in 
$\mathbb{R}^3$, that is, geodesics for the Euclidean metric. (The singular extremals for the Lorentzian case are given by null geodesics.)

\subsubsection*{Unparametrized conformal geodesics}

Consider the invariant Lagrangian $L=\tau$ corresponding to the variational problem 
\begin{equation}\label{torsion_functional}
    \int \tau \ ds = \int \tau\ \sqrt{1 + y_1^2 + z_1^2}\ dx.
\end{equation}
Computing the invariant Eulerian and invariant Hamiltonian yields 
that $\mathcal{E}_{\kappa}(L)=0$, $\mathcal{E}_{\tau}(L)=1$ and 
$\mathcal{H}(L)=-\tau$, respectively. The invariant Euler--Lagrange equations are 
\begin{equation}
    \tau \kappa = 0 \qquad \kappa_s = 0.
\end{equation}
Note that $\tau$ is not well-defined when $\kappa = 0$. The 
invariant Euler--Lagrange equations imply that both $\tau = 0$ and $\kappa_s = 
0$. The torsion of a curve vanishes if and only if it is planar, and then 
the condition $\kappa_s = 0$ implies that the curve will be a 
circle. Thus, the extremals of the variational problem \eqref{torsion_functional}
consist of circles, that is, conformal geodesics.

\subsubsection*{Elastica}

Elastica are the solution curves of the bending energy 
(or squared geodesic curvature) functional $\int \kappa^2 ds$ . 
Its invariant Euler--Lagrange equations are given by 
 \begin{equation}
\kappa_{ss} = \kappa \tau^2 - \frac12 \kappa^3, \qquad 
\kappa \tau_s + 2 \kappa_s \tau = 0.
 \end{equation}
In \cite{Singer} the invariant Euler--Lagrange equations for the bending energy functional are computed on a Riemannian manifold of constant curvature $G$ to be 
 \begin{equation}
\kappa_{ss} = \kappa \tau^2 - \frac12 \kappa^3 - G \kappa, \qquad 
\kappa \tau_s + 2 \kappa_s \tau = 0,
 \end{equation}
which are consistent with the result for the flat case $(G = 0)$. This is in contrast with the case of (unparametrized) conformal geodesics, where the invariant Euler--Lagrange equations are given by $\kappa_s = 0$ and $\kappa \tau = 0$ on any conformal 3-manifold \cite{marugame2024fefferman}, \cite{kruglikov2024variationality}. The explanation comes from the fact that the latter variational problem is conformally invariant, while the former is not.

\subsubsection{Euclidean-invariant Euler--Lagrange via global invariants}

We use the rational invariants $K:=\kappa^2, T:=\tau$, the rational invariant derivation $\nabla = \kappa \kappa_s \frac{d}{ds}$ and the same cross-section as before. For the invariant basis of 1-forms, we take 
\begin{equation}\label{eq:Eucl3D_global_frame}
    \begin{split}
        \varpi &= \frac{1}{\kappa \kappa_s\sqrt{1+y_1^2+z_1^2}}(dx+y_1 dy+z_1 dz),\\
    \vartheta_0^y &= \frac{1}{(1+y_1^2+z_1^2)^2} \left(((1+z_1^2) y_2-y_1 z_1 z_2) \theta_0^y + ((1+y_1^2)z_2-y_1 z_1 y_2) \theta_0^z \right), \\
    \vartheta_0^z &= \frac{\kappa_s}{\sqrt{(y_1z_2-z_1y_2)^2+y_2^2+z_2^2}} \left( y_2 \theta_0^z-z_2 \theta_0^y\right),
    \end{split}
\end{equation}

all of which have rational coefficients. Now the elements of 
\[ \mathcal{A}^* = \left( \begin{matrix} (\mathcal{A}_y^K)^* & (\mathcal{A}_y^T)^*\\(\mathcal{A}_z^K)^*&(\mathcal{A}_z^T)^* \end{matrix} \right), \qquad \mathcal{B} ^* =\left( \begin{matrix} (\mathcal{B}_y)^* \\ (\mathcal{B}_z)^* \end{matrix} \right)\] are 
\begin{align*}
        (\mathcal{A}_y^K)^* &= \frac{4}{K_1} \nabla^2 -\left(\frac{6 K_2}{K_1^2}-\frac{4}{K_0}\right) \nabla -\left(\frac{2K_3}{K_1^2}-\frac{4K_2^2}{K_1^3}+\frac{K_2}{K_0 K_1} + \frac{K_1}{K_0^2} +2T_0^2-2K_0\right), \\
    (\mathcal{A}_y^T)^* &= \frac{4T_0}{K_0 K_1} \nabla^2 - \left(\frac{6 T_0 K_2}{K_0 K_1^2}-\frac{2T_1}{K_0 K_1} + \frac{2 T_0}{K_0^2} \right) \nabla +\left(-\frac{2 T_0 K_3}{K_0 K_1^2}+\frac{4 T_0 K_2^2 }{K_0 K_1^3}-\frac{T_1K_2 }{K_0 K_1^2}+\frac{T_0 K_2 }{K_0^2 K_1}+2 T_0\right), \\ 
    (\mathcal{A}_z^K)^* &= \frac{8 K_0 T_0}{K_1} \nabla -\left( \frac{4K_0 T_0 K_2}{K_1^2}-\frac{4 K_0 T_1}{K_1}+4 T_0  \right), \\ 
    (\mathcal{A}_z^T)^* &= - \frac{4}{K_1^2} \nabla^3 + \left( \frac{12 K_2}{K_1^3}+\frac{4}{K_0 K_1} \right) \nabla^2 + \left( \frac{8 K_3}{K_1^3}-\frac{19  K_2^2}{K_1^4}-\frac{5  K_2}{K_0  K_1^2}+\frac{2  T_0^2}{K_1}-\frac{2  K_0}{K_1}-\frac{3}{K_0^2} \right) \\ &+ 
    \left(\frac{2 K_4}{K_1^3}-\frac{13 K_2 K_3}{K_1^4}-\frac{2 K_3}{K_0 K_1^2}+\frac{14 K_2^3}{K_1^5}+ \frac{7 K_2^2}{2K_0 K_1^3}+ \frac{3K_2}{2K_0^2 K_1}+\frac{(K_0-T_0^2 )K_2 }{K_1^2}-1 \right), \\
    (\mathcal{B}_y)^* &= \frac{4}{K_1^2} \nabla^3 + \left(-\frac{18 K_2}{K_1^3}+\frac{4}{K_0 K_1} \right) \nabla^2 +\left(-\frac{14 K_3}{K_1^3}+\frac{40 K_2^2}{K_1^4}-\frac{9 K_2}{K_0 K_1^2}-\frac{2 T_0^2}{K_1}+\frac{2 K_0}{K_1}-\frac{1}{K_0^2}\right) \nabla \\ &+ \left(-\frac{4 K_4}{K_1^3}+\frac{32 K_2 K_3}{K_1^4}-\frac{4 K_3}{K_0 K_1^2}-\frac{40 K_2^3}{K_1^5}+\frac{9 K_2^2}{K_0 K_1^3}+\frac{K_2}{K_0^2 K_1}+\frac{2 (T_0^2 -K_0)K_2 }{K_1^2}-2\right),\\
    (\mathcal{B}_z)^* &= \frac{8 K_0 T_0}{K_1^2}\nabla^2+\left( -\frac{20 K_0 T_0 K_2 }{K_1^3}+\frac{4 K_0 T_1}{K_1^2}+\frac{4 T_0}{K_1}\right) \nabla \\&+\left(-\frac{4 K_0 T_1 K_2 }{K_1^3}-\frac{8 K_0 T_0 K_3}{K_1^3}+\frac{20 K_0 T_0 K_2^2 }{K_1^4}-\frac{4 T_0 K_2 }{K_1^2}\right) .
\end{align*}
For example, if the Lagrangian is $\varpi$, the corresponding invariant Euler--Lagrange equations for regular extremals are given by  
\begin{align*}
    K_4 &= -\tfrac{-2 K_0^2 K_1^3 K_2 T_0^2+2 K_0^3 K_1^3 K_2+2 K_0^2 K_1^5-32 K_0^2 K_1 K_2 K_3+40 K_0^2 K_2^3+4 K_0 K_1^3 K_3-9 K_0 K_1^2 K_2^2-K_1^4 K_2}{4 K_0^2 K_1^2}, \\ T_1 &= -T_0 \tfrac{2 K_0 K_1 K_3-5 K_0 K_2^2+K_1^2 K_2}{K_0 K_1 K_2}.
\end{align*}

The matrix relative 
invariant $W$ determined by the global invariant forms (\ref{eq:Eucl3D_global_frame}) is 
proportional to (\ref{eq:W_Eucl3D}) with proportionality factor $\text{diag}(\kappa_s^{-1}, \kappa^{-1})$.

\subsection{Conformally Invariant Variational Problems in 3D}\label{subsConf}

The group $G=\text{SO}(1,4)$ acts on $\mathbb{S}^3$ by conformal transformations;
in fact it is the entire conformal automorphism group. 
The corresponding Lie algebra $\mathfrak{g}=\mathfrak{so}(1,4)$ 
acts algebraically in the chart $\R^3(x,y,z)$.
It is generated by three translations $\p_x$, $\p_y$, $\p_z$, three rotations 
$x\p_y-y\p_x$, $x\p_z-z\p_x $, $y\p_z-z\p_y$, , the homothety $x\p_x+y\p_y+z\p_z$, 
and three inversions $\tfrac12(x^2+y^2+z^2)\p_x-x(x\p_x+y\p_y+z\p_z)$, etc, 
enumerated as $X_1,\dots,X_{10}$.

\subsubsection{Generating differential invariants}

Local differential invariants of curves in conformal spaces go back to H.~Liebmann and T.~Takasu, we
refer to \cite{CRW,M} and references therein. In the space of constant curvature in 3D (often taken to be the sphere,
we will refer to the flat space) the generators are the following two differential invariants and invariant derivation,
expressed in terms Euclidean curvature $\kappa$, torsion $\tau$ and the natural parameter $ds$
(with $\kappa_s=\tfrac{d}{ds}\kappa$, etc):
 \begin{equation}\label{def_conf3D_inv_der}
\begin{split}
     Q & =\frac{4\nu\nu_{ss}-4\kappa^2\nu^2-5\nu_s^2}{8\nu^3},\qquad 
T=\frac{2\kappa_s^2\tau+\kappa^2\tau^3+\kappa\kappa_s\tau_s-\kappa\kappa_{ss}\tau}{\nu^{5/2}},\\
\nu & =\sqrt{\kappa^2\tau^2+\kappa_s^2},\qquad \omega=\sqrt{\nu}ds,\quad D_\omega=\nu^{-1/2}\frac{d}{ds}.
\end{split}
 \end{equation}

Thus the first absolute invariant comes in order 4 and is the conformal torsion $T$. In order 5 two more invariants add:
conformal curvature $Q$ and the derivative of torsion $T_\omega=D_\omega T$. Order 6 yields 
two more invariants $Q_\omega$, $T_{\omega\omega}$, etc. The relative invariant $\nu$ is 
called the \textit{conformal arclength}. Note a similarity of $Q$ and $q$, as well as $
\omega$ and $\sigma$.

\medskip

Let us now obtain global (rational) differential invariants.
Reference \cite{CRW} contains the following normal form of the curve in $\R^3$, 
where the first coordinate $x$ is used as a parameter:
 $$
y=\frac{x^3}{3!}+(2Q-T^2)\frac{x^5}{5!}+O(x^6),\quad z=T\frac{x^4}{4!}+T_\omega\frac{x^5}{5!}+O(x^6).
 $$
Note that for unparametrized curve several choices were made here: orientation and co-orientations 
(including the orientation of the ambient space), which are not invariant under the action of the conformal group
$O(4,1)$. This is a rather a pseudo-group (or groupoid) action on $\R^3$, while the group acts on $\mathbb{S}^3$.

The above coordinate choice is respected by the transformation
$(x,y,z)\mapsto\Bigl(\sigma x,\sigma^{-1}y,\epsilon\sigma z\Bigr)$, where $\epsilon^2=1$, $\sigma^4=1$. 
Taking into account this discrete freedom the local invariants change as follows:
 $$
T\mapsto\epsilon\sigma T,\qquad Q\mapsto\sigma^2Q,\qquad D_\omega\mapsto\sigma^{-1} D_\omega,
 $$
Note also that the Euclidean curvature and the natural parameter change $\kappa\mapsto\sigma^2\kappa$,
$\tfrac{d}{ds}\mapsto\sigma^2\tfrac{d}{ds}$, while $\tau$ is invariant. 
Thus the expressions containing even number of $\kappa$ and $s$-derivatives are invariant, and so the 
global rational conformal invariants have the following generators:
 \begin{equation}
T^4,\qquad QT^2,\qquad \nabla=TT_\omega D_\omega. 
 \end{equation}
The first invariant is $T^4$ of order 4, then come $Q^2$, $QT^2$ and $T_\omega^2$ of order 5, but we note that 
$Q^2=(QT^2)^2/T^4$ is derived. In order 6 we get $T^{-1}T_\omega Q_\omega$ and $TT_{\omega\omega}$, etc. 
The lowest order global invariant form is $T^{-1}T_\omega^{-1}\omega$.

 \begin{remark}
While we indicate real Lie group, our residual freedom is considered
with respect to complex variables, and so is applicable to any
signature of the conformal structure. Even though the restriction may be 
simplified to $\sigma^2=1$ over $\R$, we keep the above constraint to 
get rational invariants, cf.\ \cite[\S5.3]{kruglikov2016global}.
 \end{remark}

\subsubsection{Conformally-invariant Euler--Lagrange via local invariants}

Using the action of $G=\text{Conf}(3)$ to fix as many group parameters 
as possible, we obtain the following cross-section
 \begin{equation}
\Sigma := \{x=0,\ y=0,\ z=0,\ y_1=0,\ z_1=0,\ y_2=0,\ z_2=0,\ y_3=1,\ z_3=0,\ y_4=0 \}.
 \end{equation}
Restricting the conformal torsion $T$ and conformal curvature $Q$ to the 
cross-section gives $T|_\Sigma = z_4, Q|_\Sigma = \frac12 (z_4^2 + y_5)$. 
Thus, we have that $T = \iota(z_4)$ and $Q = \frac12 \iota(y_5 + z_4^2)$. 
Moreover, the horizontal form $\omega$ (dual to $D_{\omega}$) equals $dx$ on the 
cross-section. Therefore, we define $\varpi := \iota(dx)$, which ensures that 
the horizontal part of $\varpi$ is equal to $\omega$. We compute the invariant 
Euler--Lagrange equations by application of the algorithm.

We get the following expressions for the $\gamma^i$ and $\epsilon^i$ components:
\begin{equation}
    \begin{split}
        \gamma^1 & = - \varpi,\ \gamma^2 = 0, \gamma^3 = 0,\ \gamma^4 = 0,\ \gamma^5 = 0,\ \gamma^6 =  -T\, \varpi, 
        \gamma^7  = 0,\ \gamma^8 = - Q\, \varpi, \gamma^9 = - \varpi,\ \gamma^{10} = 0, \\
        \epsilon^1 & = 0,\ \epsilon^2 = -\iota(\theta^y_0), \epsilon^3 = -\iota(\theta^z_0),\ \epsilon^4 = -\iota(\theta^y_1),\ \epsilon^5 = -\iota(\theta^z_1),\ \epsilon^6 = -\iota(\theta^z_3), \\
        \epsilon^7 & = \frac{\iota(\theta^y_3)}{2},\ \epsilon^8 = -\frac12\ \iota(\theta^y_4) - \frac{T \iota(\theta_3^z)}{2} , \epsilon^9 = -\iota(\theta^y_2),\ \epsilon^{10} = -\iota(\theta^z_2).
    \end{split}
\end{equation}

By computing the invariant vertical differentials $d_{\mathcal{V}}T,d_{\mathcal{V}}Q$ 
and $d_{\mathcal{V}}\varpi$ and rewriting these as the actions of invariant 
differential operators on the zeroth-order invariant contact forms, we get the 
expressions for $\mathcal{A}$ and $\mathcal{B}$. We readily compute their formal 
adjoints as
\begin{equation}\label{ABconf3D}
    \begin{split}
        (\mathcal{A}_y^T)^{*} & = -\frac{5 T}{2} D_{\omega}^3 - \frac{3 T_{1}}{2} D_{\omega}^2 +\left(\frac{1}{2} T_{2} - \frac{1}{2} T^{3} + T Q \right) D_{\omega}+\left( \frac{1}{2} T_{3} -3 T^{2} T_{1} - \frac{1}{2} T Q_{1} - T_{1} Q  \right)    \\
        (\mathcal{A}_z^T)^{*} & = D_{\omega}^{4}+\left(-\frac{3 T^{2}}{2} - 2 Q\right) D_{\omega}^{2}+ \left(-Q_{1}+\frac{3 T T_{1}}{2}\right) D_{\omega} +\left(2 T T_{2}+\frac{3}{2} T_{1}^{2} -1 - \frac{1}{2} T^{4} -  Q T^{2}\right) \\
        (\mathcal{A}_y^Q)^{*} & = -\frac{1}{2} D_{\omega}^5 + \left(\frac{3 T^{2}}{2}+2 Q\right) D_{\omega}^3 +\left(\frac{3 T T_{1}}{2}+\frac{7 Q_{1}}{2}\right) D_{\omega}^2+\left(-3 Q T^{2}-2 Q^{2}+3 Q_{2}-2\right) D_{\omega}  \\ &  +\left(Q_3 -3 Q T T_{1}-3 Q_{1} T^{2}-3 Q Q_{1}\right) \\
        (\mathcal{A}_z^Q)^{*} & = -\frac{3 T}{2} D_{\omega}^4 -\frac{3 T_{1} }{2} D_{\omega}^3+\left(4 T Q - \frac{1}{2} T_{2}+\frac{1}{2} T^{3}\right) D_{\omega}^2+\left(3 T_{1} Q + 6 T Q_{1}\right) D_{\omega} \\ & +\left(-Q T^{3}-2 Q^{2} T + T_{2} Q + 3 T_{1} Q_{1} + 3 T Q_{2} - T\right) \\
        (\mathcal{B}_y)^{*} & =  -\frac1{2} D_{\omega}^3 +\left(\frac{3 T^{2}}{2}+Q\right) D_{\omega} + \left(\frac{Q_{1}}{2}+\frac{3 T T_{1}}{2}\right)  \\
        (\mathcal{B}_z)^{*} & = -\frac{3 T}{2} D_{\omega}^2 -\frac{3 T_{1}}{2} D_{\omega} + \left(-\frac{1}{2} T_{2}+ T Q+\frac{1}{2} T^{3}\right)
    \end{split}
\end{equation}

We compute some examples of conformally invariant Euler--Lagrange equations
based on \eqref{ABconf3D}. The results are displayed in the following table.

\medskip

\begin{tabular}{||c c c c||} 
 \hline
 Lagrangian & Eulerian & Hamiltonian & Invariant Euler--Lagrange Equations \\ [0.5ex] 
 \hline\hline
 $\omega$ & $ (0,0)^T $ & $-1$ & $  \begin{cases}
     Q_1 + 3 T T_1 = 0\\
     T_2 =  T^3 + 2 T Q
    \end{cases}  $ \\ 
 \hline
  $T\ \omega$ & $(1,0)^T  $ & $ -T $ & $ \begin{cases}
     0 = 0\\
     0 = -1
    \end{cases} $ (\textbf{NB}: This implies no \textit{regular} extremals.)\\
 \hline
 $Q\ \omega$ & $(0,1)^T$ & $ -Q $ & $  \begin{cases}
    Q_3 = 3 T Q T_1+ 3 Q T^2 + 3 Q Q_1 
    \\
     Q_2 = \frac{1}{3T} \left(T^3 Q + (2 Q^2 +2)T - T_2 Q - 3 T_1 Q_1   \right)
    \end{cases}  $\\
    \hline
     $T^2\ \omega$ & $ (2 T, 0)^T $ & $ - T^2  $ & $ \begin{cases}
     T_4 = \frac14 T^5 + \frac12 T^3 Q+ \frac54 T^2 T_2+ 2 Q T_2 + T_1 Q_1+ T
      \\
     T_3 = \frac{1}{10 T} \left(T_1\, (-10 T_2 + 4 T Q - 5 T^3) - T^2 Q_1 \right) 
    \end{cases} $ \\
 \hline
 $Q^2\ \omega$ & $(0, 2Q)^T$ & $  -Q^2  $ & $ \begin{cases}
     Q_5 = R(T, T_1, Q_1,Q_2,Q_3) \\
     Q_4 = \frac{S(T, T_1, T_2, Q,Q_1,Q_2,Q_3)}{6 T}
    \end{cases}  $  ($R,S$ polynomials)    \\  
 \hline
\end{tabular}

\medskip

Here the polynomials $R,S$ are given by the following formulas:
\begin{equation}
    \begin{split}
        R & =  -\frac{9}{2} Q^{2} T T_{1}-9 Q Q_{1} T^{2}+3 Q_{2} T T_{1}+3 Q_{3} T^{2}-\frac{15}{2} Q^{2} Q_{1}+5 Q_{0} Q_{3}+10 Q_{2} Q_{1}-4 Q_{1} \\
        S & = -3 Q^{2} T^{3}-6 Q^{3} T+2 Q_{2} T^{3}+3 Q^{2} T_{2}+18 T_{1} Q Q_{1}+22 Q Q_{2} T+18 Q_{1}^{2} T-4 T Q-2 T_{2} Q_{2}-6 T_{1} Q_{3}
    \end{split}
\end{equation}

\subsubsection{Computing relative invariant $W$ using a general invariant coframe}

In order to relate the invariant Euler--Lagrange 
equations to the usual Euler--Lagrange equations, it is essential to compute the matrix 
relative invariant $W$. However, for the above computations, we did not have to make an 
appeal to explicit expressions for the invariant contact one-forms. 
Computing the explicit moving frame is computationally difficult. In 
order to circumvent this challenge, we use the invariant contact forms of order 0 
constructed in \eqref{eq:vartheta0}, which we can relate to the invariantizations of the 
usual contact forms. This will allow us to compute $W$ and in particular determine the Euler--Lagrange equations for the Lagrangian $T\ \omega$.

\medskip

By application of formulas (\ref{eq:varthetatilde}), (\ref{eq:vartheta0}), 
we construct from $Q, T$ two invariant 
contact forms $ \vartheta_0^1, \vartheta_0^2$ of order~0:
\begin{equation}
    \begin{split}
        \vartheta^1_0 := 2 \left( \frac{dQ(\partial_{y_5})}{(D_{\omega}(x))^5} \cdot \theta_0^y +  \frac{dQ(\partial_{z_5})}{(D_{\omega}(x))^5} \cdot \theta_0^z    \right)   \hspace{1cm}
        \vartheta^2_0  := \left(  \frac{dT(\partial_{y_4})}{(D_{\omega}(x))^4} \cdot \theta_0^y +  \frac{dT(\partial_{z_4})}{(D_{\omega}(x))^4} \cdot \theta_0^z  \right).
    \end{split}
\end{equation}

We want to relate these one-forms to the invariantizations $
\iota(\theta_0^y)$ and $\iota(\theta_0^z)$. To this end, we evaluate $\vartheta_0^1, 
\vartheta_0^2$ to the cross-section $\Sigma$. By definition of the 
invariantization map, we have that $\iota(\theta_0^y)|_\Sigma = \theta_0^y$ 
and $\iota(\theta_0^z)|_\Sigma = \theta_0^z$. We obtain coordinate 
expressions of the form
\begin{equation*}
    (\vartheta^i_0)|_\Sigma = a^i_y \, \theta_0^y + a^i_z \, \theta_0^z \qquad (i = 1,2). 
\end{equation*}
In this example the coefficients happen to be
\begin{equation*}
    a^1_y = 1,\quad a^1_z = 0, \qquad a^2_y = 0,\quad a^2_z = 1,
\end{equation*}
and therefore we have that $\vartheta^1_0 = \iota(\theta^y_0)$ and $
\vartheta^2_0 = \iota(\theta^z_0)$. Thus, we 
have found explicit coordinate expressions for the invariantized contact 
forms of order 0 
without using the explicit moving frame. Since the horizontal part of $
\varpi$ is known (\ref{def_conf3D_inv_der}), we can then go on to compute 
$W$ using its definition (\ref{eq:W}). We obtain the following expression (note the transposition):

\begin{equation}\label{eq:WConformal3D}
\hspace{-5pt}W = \scalebox{0.95}{$
    \begin{pmatrix}
\frac{(z_1^2+1)y_3-y_1z_1z_3}{\ell^4}
-3\frac{(y_1y_2+z_1z_2)((z_1^2+1)y_2-y_1z_1z_2)}{\ell^6} & 
-\frac{z_3}{\ell^3}-3\frac{(y_1y_2+z_1z_2)z_2}{\ell^5} \\
\frac{(y_1^2+1)z_3-y_1z_1y_3}{\ell^4}
-3\frac{(y_1y_2+z_1z_2)((y_1^2+1)z_2-y_1z_1y_2)}{\ell^6} & 
\frac{y_3}{\ell^3}-3\frac{(y_1y_2+z_1z_2)y_2}{\ell^5}
\end{pmatrix}^T $}\!\!.\hspace{-5pt}
\end{equation} 

It is readily verified that $W$ is a matrix relative invariant. 
(\textbf{NB}: This does not imply that its entries are scalar relative 
invariants.) 

\subsubsection{Singular Extremals: Conformal Geodesics}\label{singular_extremals_conf_geod}

In \cite{kogan2003invariant} the authors raise the question if there are 
examples of invariant variational problems all of whose extremals are 
singular. We show that the conformal geodesics on $\mathbb{R}^3$ provide 
such an example. It is known that the conformal geodesics are variational 
with the Lagrangian $\tau\ ds$, which is invariant with respect to the isometry 
group of $\mathbb{R}^3$. In \cite{magliaro2011geometry} it is shown that 
\begin{equation}
    \int \tau\ ds = \int T\ \omega\hspace{1cm} (\text{mod}\ 2 \pi \mathbb{Z}),
\end{equation}
where the integrals are taken over paths with fixed ends,
and so the conformal geodesics are also extremals of the conformal torsion 
functional. The conformal geodesic equations $\kappa_s =0$ and $\kappa\tau=0$ are 
equivalent to the vanishing of the conformal relative invariant $\nu = \sqrt{\kappa_s^2 + 
\kappa^2 \tau^2}$. Since $\nu$ appears in the denominator of the conformal 
torsion $T$, we conclude that the conformal geodesics are singular 
extremals.

\medskip

We can still compute the Euler--Lagrange equations using Theorem 1. We 
obtain that the invariant Euler--Lagrange system is given by $\text{E}_{\mathrm{inv}}(T) = (0,-1)^{T}$. Setting this equal to zero gives of course no solutions, which means that there are no regular extremals.
The matrix relative invariant $W$ is readily computed (\ref{eq:WConformal3D}) and satisfies $\det(W) = \sqrt{1 + y_1^2 + z_1^2}\ \nu^2$. Solving the Euler--Lagrange 
equation $\mathrm{E}(L) = W \cdot (0, -1)^{T} = 0$ for third-derivatives yields
\begin{equation}
     y_3 = \frac{3 y_2(y_1 y_2 + z_1 z_2)}{1 + y_1^2 + z_1^2}  \qquad z_3 = \frac{3z_2(y_1 y_2 + z_1 z_2)}{1 + y_1^2 + z_1^2}.
\end{equation}
These are readily recognized to be the unparametrized conformal geodesic equations. 
Alternatively, one can note that $\det(W) = 0$ by vanishing of the second column of $W$. 
This implies that $\nu = 0$ and so we obtain the two equations $\kappa_s = 0$ and $\kappa 
\tau = 0$ which describe (unparametrized) conformal geodesics. 
\medskip

 \begin{remark}
This example also shows that it may be convenient to restrict to a subgroup 
$H \subseteq G$ when computing Euler--Lagrange equations for some $G$-invariant variational problem. In this case, when restricting to $\text{Iso}(\mathbb{R}^3) \subseteq \text{Conf}(\mathbb{R}^3)$ we still have two 
generating differential invariants, but these are of lower order. Moreover, 
the set of singular points becomes smaller in this case. Actually, the conformal geodesics are regular extremals for the isometry-invariant variational problem $\int \tau\ ds$.
 \end{remark}

\subsubsection{Conformal Arclength}

If we consider the Lagrangian $L = \omega$, where $\omega = \nu\, dx$ is the 
conformal arclength, we obtain the two equations $Q_1 + 3 T T_1 = 0$ and $T_2 =T^3 + 2 T 
Q $. These equations are both of sixth order and coincide with the expressions in 
\cite{M}, \cite{magliaro2011geometry}.

\textbf{NB}: 
Extremals of the conformal arclength functional are sometimes also called conformal geodesics in the literature cf. \cite{M}, \cite{magliaro2011geometry}. 
These are however two conflicting nomenclatures.

\subsubsection{An example of globally invariant variational problem}

Consider the global invariant Lagrangian $\lambda=T^{-1}T^{-1}_{1}\omega$ 
dual to $\nabla = TT_{1} D_{\omega}$. We base on Remark \ref{Rk26} 
and subsection \ref{S27} to compute the extremals. 
Thus for $L=T^{-1}T_{1}^{-1}$ the  Eulerian is given 
by $\mathcal{E}(L) = -\frac{2}{T^2T_1} -\frac{2T_2}{TT_1^3}$, 
and the Hamiltonian by $\mathcal{H}(L) = -\frac{2}{TT_1}$. 
From this we obtain the invariant Euler--Lagrange equations via local differential
invariants, but these are readily expressed in terms of global invariants. 

Let $P:= T^4$, $R:= QT^2$ be 
a basis of global differential invariants; let us use subscripts to indicate  
the number of derivatives with respect to $\nabla$. 
We obtain two equations that can be expressed in terms of global invariants
and resolved with respect to $P_5$ and $R_2$:
 \begin{align*}
& 160P^4P_1^3P_5 + 8P^4P_1^4P_3 - 20P^4P_1^3P_2^2 + 28P^3P_1^5P_2 
  + 8P^3P_1^4P_2R_1 - 16P^3P_1^4P_3R + 40P^3P_1^3P_2^2R \\
& \ \ - 12P^2P_1^7 - 8P^2P_1^6R_1 - 16P^2P_1^5P_2R - 1840P^4P_1^2P_2P_4 
  - 1280P^4P_1^2P_3^2 + 10560P^4P_1P_2^2P_3 - 8800P^4P_2^4 \\
& \ \ + 656P^3P_1^4P_4 - 4868P^3P_1^3P_2P_3 + 5748P^3P_1^2P_2^3 + 306P^2P_1^5P_3 
  - 654P^2P_1^4P_2^2 - 20PP_1^6P_2 + 5P_1^8 =0,\\[7pt]
& 160P^3P_1^3(PP_2 - P_1^2)R_2 + 100P^5P_1^4P_2 - 100P^4P_1^6 + 200P^4P_1^4P_2R 
  - 200P^3P_1^6R + 1360P^5P_1^3P_4 \\
& \ \ - 10880P^5P_1^2P_2P_3 + 13600P^5P_1P_2^3 + 200P^4P_1^4P_2 + 3224P^4P_1^4P_3 
  - 7380P^4P_1^3P_2^2 + 640P^4P_1^3P_3R_1 \\
& \ \ + 1280P^4P_1^3P_4R - 1680P^4P_1^2P_2^2R_1 - 10240P^4P_1^2P_2P_3R 
  + 12800P^4P_1P_2^3R - 100P^3P_1^6 \\
& \ \ + 389P^3P_1^5P_2 + 944P^3P_1^4P_2R_1 + 3072P^3P_1^4P_3R 
  - 6840P^3P_1^3P_2^2R - 41P^2P_1^7 + 56P^2P_1^6R_1 \\
& \ \ + 112P^2P_1^5P_2R - 240P^4P_1^2P_2P_4 - 400P^4P_1^2P_3^2 
  + 2600P^4P_1P_2^2P_3 - 2200P^4P_2^4 + 48P^3P_1^4P_4 \\
& \ \ + 116P^3P_1^3P_2P_3 - 786P^3P_1^2P_2^3 - 792P^2P_1^5P_3 
  + 1853P^2P_1^4P_2^2 + 15PP_1^6P_2 - 235P_1^8 = 0.
 \end{align*}

\subsection{Euclidean Invariant Variational Problems in $4D$}

The manifold is $\R^4(x,y,z,u)$. We consider the Lie group action of $\text{SE}(4)$ 
on $\R^4$. For the Lie algebra $\mathfrak{se}(4)=\mathfrak{so}(4)\ltimes\R^4$ 
we have generators $\p_x$, $\p_y$, $\p_z$, $\p_u$, $x\p_y-y\p_x$, $x\p_z-z\p_x$,
$x\p_u-u\p_x$, $y\p_z-z\p_y$, $y\p_u-u\p_y$, $z\p_u-u\p_z$, which we denote by $X_1,\dots,X_{10}$, respectively. In the case of Minkowski spacetime
the algebra is $\mathfrak{se}(3,1)=\mathfrak{so}(3,1)\ltimes\R^4$ and the 
generators $X_7,X_9,X_{10}$ should be changed to hyperbolic rotations;
this modifies the formulas below by some sign changes, but otherwise keeps them 
true (including examples). The neutral signature works similarly. 

\subsubsection{Generating differential invariants}
Denote by $\mathsf{W}(f,g,h)$ the Wronskian for three jet functions $f,g,h$. 
Then the algebra of local differential invariants is generated by
 \begin{gather*}
\kappa=\frac{\sqrt{(y_1z_2-z_1y_2)^2+(y_1u_2-u_1y_2)^2+(z_1u_2-u_1z_2)^2+y_2^2+z_2^2+u_2^2}}{(1+y_1^2+z_1^2+u_1^2)^{3/2}},\\
\tau=\frac{\left(\mathsf{W}(y_1,z_1,u_1)^2+(y_2z_3-z_2y_3)^2+(y_2u_3-u_2y_3)^2+(z_2u_3-u_2z_3)^2\right)^{1/2}}{(y_1z_2-z_1y_2)^2+(y_1u_2-u_1y_2)^2+(z_1u_2-u_1z_2)^2+y_2^2+z_2^2+u_2^2}\\
\mu= \frac{1}{\kappa^3 \tau^2} \frac{\mathsf{W}(y_2, z_2, u_2)}{(1+y_1^2+z_1^2+u_1^2)^5},\qquad\qquad
\frac{d}{ds}=\frac1{\sqrt{1+y_1^2+z_1^2+u_1^2}}\frac{d}{dx}.
 \end{gather*}

Hence generators of the algebra of global differential invariants are
 $$
\kappa^2,\quad \tau^2,\quad \kappa^3\tau^2\mu,\quad \nabla=\kappa\kappa_s\frac{d}{ds}.
 $$

\subsubsection{Computing invariant Euler--Lagrange equations using cross-section}
We apply the algorithm as described in section (\ref{algorithm}). A cross-section is 
given by the vanishing of $x, y_0, y_1, z_1, z_2, u_1, u_2,u_3$.
By applying the invariant horizontal differential to these functions which vanish on the 
cross-section, one readily obtains that the nonvanishing $\gamma^i$ are given by
\begin{equation}\label{4D_eqn_gammas}
    \gamma^1 = - \varpi,\ \gamma^5 = - \kappa \varpi,\ \gamma^8 = -\tau \varpi,\ \gamma^{10} = -\mu \varpi.
\end{equation}
The $\epsilon^i$ components are given by
\begin{equation}
    \begin{split}
        \epsilon^1 & = 0,\ \epsilon^2 = -\iota(\theta_0^y),\ \epsilon^3 = -\iota(\theta_0^z),\ \epsilon^4 = -\iota(\theta_0^u),\ \epsilon^5 = -\iota(\theta_1^y),\ \epsilon^6 = -\iota(\theta_1^z),     \\ 
        \epsilon^7 & = -\iota(\theta_1^u),\, \epsilon^8 = -\frac{\iota(\theta_2^z)}{\kappa},\ \epsilon^9 = -\frac{\iota(\theta_2^u)}{\kappa},\ \epsilon^{10} = -\frac{\iota(\theta_3^u)}{\kappa \tau} + \frac{\kappa_1 \iota(\theta_2^u)}{\kappa^2 \tau}.
    \end{split}
\end{equation}
Evaluating $\kappa, \tau, \mu$ and the contact-invariant horizontal form $ds$ on the 
cross-section yields $\kappa|_\Sigma = y_2,\ \tau|_\Sigma = \frac{z_3}{y_2}$, $\mu|_\Sigma = \frac{u_4}{z_3}$ and $ds|_\Sigma = dx$, respectively. Thus, we have that $\kappa = \iota(y_2),\ \tau = \iota\left(\frac{z_3}{y_2}\right), \mu = \iota\left(\frac{u_4}{z_3}\right)$. We define $
\varpi := \iota(dx)$ so that the horizontal part of $\varpi$ equals $ds$.

The formal adjoints of the operators $\mathcal{A}$ coming from $d_{\mathcal{V}}\kappa,d_{\mathcal{V}}\tau, d_{\mathcal{V}}\mu $ are readily computed:
\begin{equation}
    \begin{split}
        (\mathcal{A}_y^{\kappa})^{*} & = D_s^2 + (\kappa^2 - \tau^2)  \\
        (\mathcal{A}_z^{\kappa})^{*} & = 2 \tau D_s + \tau_1    \\
        (\mathcal{A}_u^{\kappa})^{*} & = \tau \mu
    \end{split}
\end{equation}
Next, we get the following expressions
\begin{equation}
    \begin{split}
        (\mathcal{A}_y^{\tau})^{*} & = \left(2\,{\frac {\tau}{\kappa}} \right) D_s^2 + \left(-2\,{\frac {\tau\,\kappa_1}{{\kappa}^{2}}}+{\frac {\tau_1}{\kappa}} \right) D_s + \left(2\,\tau\,\kappa-{\frac {\tau\,{\mu}^{2}}{\kappa}} \right) \\
        (\mathcal{A}_z^{\tau})^{*} & = \left(\frac{-1}{\kappa} \right) D_s^3 + \left(2\,{\frac {\kappa_1}{{\kappa}^{2}}} \right) D_s^2 + \left(-\kappa+{\frac {{\tau}^{2}}{\kappa}}+3\,{\frac {{\mu}^{2}}{\kappa}}+{
\frac {\kappa_2}{{\kappa}^{2}}}-2\,{\frac {{\kappa_1}^{2}}{{\kappa}^{3}}
}
 \right) D_s + \left(-\kappa_1-2\,{\frac {\kappa_1\,{\mu}^{2}}{{\kappa}^{2}}}+3\,{\frac {\mu
\,\mu_1}{\kappa}}
 \right)  \\
        (\mathcal{A}_u^{\tau})^{*} & = \left(-3\,{\frac {\mu}{\kappa}} \right) D_s^2 + \left(4\,{\frac {\kappa_1\,\mu}{{\kappa}^{2}}}-3\,{\frac {\mu_1}{\kappa}} \right) D_s + \left(2\,{\frac {\kappa_1\,\mu_1}{{\kappa}^{2}}}+{\frac {\mu\,\kappa_2}{{\kappa
}^{2}}}+{\frac {{\mu}^{3}}{\kappa}}-{\frac {\mu_2}{\kappa}}-\kappa\,\mu
-2\,{\frac {\mu\,{\kappa_1}^{2}}{{\kappa}^{3}}}
 \right)
    \end{split}
\end{equation}
Finally, the remaining formal adjoints are given by the following formulas
\begin{equation}
    \begin{split}
        (\mathcal{A}_y^{\mu})^{*} & = \frac{3 \mu}{\kappa} D_s^2 +\left(-\frac{2 \kappa_1 \mu}{\kappa^{2}}-\frac{2 \mu \tau_1}{\kappa \tau}+\frac{2 \mu_1}{\kappa}\right) D_s +\left(\frac{\mu \tau^{2}}{\kappa}+\mu \kappa\right)  \\
        (\mathcal{A}_z^{\mu})^{*} & = -\frac{3 \mu D_s^3}{\kappa \tau}+\left(-\frac{3 \mu_1}{\kappa \tau}+\frac{4 \kappa_1 \mu}{\kappa^{2} \tau}+\frac{6 \tau_1 \mu}{\kappa \tau^{2}}\right) D_s^2+\left(-\frac{\mu \kappa}{\tau}+\frac{\mu^{3}}{\kappa \tau}-\frac{\mu \tau}{\kappa}-\frac{\mu_2}{\kappa \tau}-\frac{4 \kappa_1 \tau_1 \mu}{\kappa^{2} \tau^{2}} \right. \\ & \left. +\frac{3 \tau_1 \mu_1}{\kappa \tau^{2}}+\frac{2 \kappa_1 \mu_1}{\kappa^{2} \tau}+\frac{\kappa_2 \mu}{\kappa^{2} \tau}-\frac{2 \kappa_1^{2} \mu}{\kappa^{3} \tau}-\frac{6 \tau_1^{2} \mu}{\kappa \tau^{3}}+\frac{3 \mu \tau_2}{\kappa \tau^{2}}\right) D_s+\left(-\frac{2 \mu \tau_1}{\kappa}-\frac{\mu_1 \tau}{\kappa}+\frac{2 \kappa_1 \tau \mu}{\kappa^{2}}\right) \\
        (\mathcal{A}_u^{\mu})^{*} & = \frac{D_s^4}{\kappa \tau}+\left(-\frac{3 \tau_1}{\kappa \tau^{2}}-\frac{2 \kappa_1}{\kappa^{2} \tau}\right) D_s^3 +\left(-\frac{3 \mu^{2}}{\kappa \tau}-\frac{\kappa_2}{\kappa^{2} \tau}+\frac{2 \kappa_1^{2}}{\kappa^{3} \tau}+\frac{6 \tau_1^{2}}{\kappa \tau^{3}} -\frac{3 \tau_2}{\kappa \tau^{2}}+\frac{\tau}{\kappa}+\frac{4 \kappa_1 \tau_1}{\kappa^{2} \tau^{2}}+\frac{\kappa}{\tau}\right) D_s^2 \\ & +\left(-\frac{2 \kappa_1 \tau}{\kappa^{2}}-\frac{\kappa \tau_1}{\tau^{2}}-\frac{6 \tau_1^{3}}{\kappa \tau^{4}}-\frac{\tau_3}{\kappa \tau^{2}}+\frac{2 \tau_1}{\kappa}-\frac{3 \mu \mu_1}{\kappa \tau}+\frac{2 \kappa_1 \mu^{2}}{\kappa^{2} \tau}+\frac{3 \tau_1 \mu^{2}}{\kappa \tau^{2}}-\frac{2 \kappa_1^{2} \tau_1}{\kappa^{3} \tau^{2}}+\frac{\kappa_2 \tau_1}{\kappa^{2} \tau^{2}}-\frac{4 \kappa_1 \tau_1^{2}}{\kappa^{2} \tau^{3}} \right. \\ & \left. +\frac{2 \kappa_1 \tau_2}{\kappa^{2} \tau^{2}}+\frac{6 \tau_1 \tau_2}{\kappa \tau^{3}}+\frac{\kappa_1}{\tau}\right) D_s +\left(-\frac{\mu^{2} \tau}{\kappa}+\frac{2 \kappa_1^{2} \tau}{\kappa^{3}}-\frac{\tau \kappa_2}{\kappa^{2}}-\frac{2 \kappa_1 \tau_1}{\kappa^{2}}+\frac{\tau_2}{\kappa}\right) 
    \end{split}
\end{equation}

Finally, by computing $d_{\mathcal{V}}\varpi$, we get that 
\begin{equation}
    \mathcal{B}_y = - \kappa, \mathcal{B}_z = 0, \mathcal{B}_u = 0,
\end{equation}
and consequently we get $\mathcal{B}^{*} = \mathcal{B}^T$.

The invariant Euler--Lagrange equations are given by the vector equation 
$\mathcal{A}^{*}  \mathcal{E}(L) - \mathcal{B}^{*} \mathcal{H}(L) = 0$,
where $\mathcal{E}(L), \mathcal{H}(L)$ are the invariant Eulerian and invariant Hamiltonian (\ref{def_inv_Eulerian_Hamiltonian}). The matrix relative invariant $W$ is readily computed from its definition and satisfies $\det{W} = \ell^2$. Thus, singular extremals do not exist for Euclidean signature, whereas for Lorentzian signature singular extremals are necessarily null.

Some $\mathrm{SE}(4)$-invariant Euler--Lagrange equations are displayed in the following table.

\medskip

\begin{tabular}{||c c c c||} 
 \hline
 Lagrangian & Eulerian & Hamiltonian & Invariant Euler--Lagrange Equations \\ [0.5ex] 
 \hline\hline
 $ds$ & $ (0,0,0)^T $ & $-1$ & $ \kappa = 0 $ \\ 
 \hline
 $\kappa\, ds$ & $(1,0,0)^T$ & $ -\kappa $ & $ \tau = 0  $\\
 \hline
 $\tau\, ds$ & $(0,1,0)^T  $ & $ -\tau $ & $ \mu^2 = \kappa^2 $ \\
 \hline
 $\mu\, ds$ & $(0,0,1)^T  $ & $ -\mu $ & $ \begin{cases}
     \mu = 0\\
     \tau_2 = \frac{\left(2 \kappa \kappa_1 \tau_1 + \kappa \kappa_2 \tau - 2 \kappa_1^2 \tau \right)}{\kappa^2} 
    \end{cases}  $ \\
 \hline
  $\kappa^2 \, ds$ & $(2\kappa,0,0)^T  $ & $ -\kappa^2 $ & $ \begin{cases}
     \kappa_2 = \kappa \tau^2 - \frac12 \kappa^3 \\
     \kappa \tau_1 + 2 \kappa_1 \tau = 0 \\
     \mu = 0
    \end{cases} $ \\
 \hline
 $(1 + \kappa)\ ds$ & $(1,0,0)^T$ & $ -(1 + \kappa) $ & 
 $\begin{cases} \kappa = - \tau^2\\ \tau_1=0\\ \mu = 0 \end{cases}$ \\  
 \hline
\end{tabular}

\bigskip

Let us comment on some entries. 

For $L=1$ the result follows from $E_{\text{inv}}=(-\kappa,0,0)$. 
For $L=\kappa$ we get $E_{\text{inv}}=(-\tau^2,\tau_1,\tau\mu)$,
so the Euler--Lagrange equations are equivalent to one equation $\tau = 0$. 
For $L=1+\kappa$ we get superposition
$E_{\text{inv}}(1+\kappa)=E_{\text{inv}}(1) + E_{\text{inv}}(\kappa) =
(-\kappa-\tau^2,\tau_1,\tau\mu)$, whence the result.
(We remark that formally one also gets the branch $\kappa=0,\tau=0$ but whenever
$\kappa=0$, $\tau$ is not defined, so this branch is excluded.)
Note that the equations do not show superposition, though the corresponding
invariant Euler--Lagrange operators do.

For the invariant variational problem $\int \tau\, ds$ the Euler--Lagrange 
equations are equivalent to $\mu=\pm\kappa$ or $\tau=0$. 
(The equations corresponding to $z, u$ consist of the differential consequences 
$\mu_1=\pm\kappa_1$ and $\mu_2=\pm\kappa_2$. Since $\mu$ is not defined when 
$\tau = 0$, some care must be taken when interpreting the Euler--Lagrange equations.) 

\medskip

In $4D$ the equations for the elastica, that is, solution curves of the invariant variational problem $\int \kappa^2\, ds$ are essentially the same as in $3D$. Indeed, the equation $\mu = 0$ implies that the solution curves lie on a hyperplane in $\mathbb{R}^4$ on which we obtain again the two equations $\kappa_2 = \kappa \tau^2 - \frac12 \kappa^3$ and $\kappa \tau_1 + 2 \kappa_1 \tau = 0$.

\subsubsection{An example of globally invariant variational problem}

Consider the global invariant Lagrangian $\lambda=\kappa^{-1}\kappa_1^{-1}\,ds$ 
dual to $\nabla = \kappa\kappa_1D_s$. We base on Remark \ref{Rk26} 
and subsection \ref{S27} to compute the extremals. 
Thus for $L=\kappa^{-1}\kappa_1^{-1}$ the Eulerian is given by 
$\mathcal{E}(L) = -\frac{2}{\kappa^2\kappa_1} -\frac{2\kappa_2}{\kappa\kappa_1^3}$, 
and the Hamiltonian by 
$\mathcal{H}(L) = -\frac{2}{\kappa\kappa_1}$. 
From this we obtain the invariant Euler--Lagrange equations via local differential
invariants, but these are readily expressed in terms of global invariants. 

Let $P:= \kappa^2$, $R:= \tau^2$ and $S:= \kappa^3\tau^2\mu$ be 
a basis of global differential invariants; let us use subscripts to indicate 
the number of derivatives with respect to $\nabla$. 
We obtain three equations that can be solved for $\kappa_4$, $\tau_1$, $\mu$. 
By expressing local invariants in terms of global invariants, we obtain the following equations:
\begin{equation}
    \begin{cases} 
P_4 = 
\frac{P_1^3P_2R+P_1^5+16P_1P_2P_3-20P_2^3}{2P_1^2}
+\frac{9PP_2^2 -2P^3P_1P_2 -4PP_1P_3 +P_1^2P_2}{4P^2}, \\
R_1 = -\frac{2R(2PP_1P_3-5PP_2^2+P_1^2P_2)}{PP_1P_2}, \\
S = 0.\end{cases}
\end{equation}

\section{Outlook: theory vs symbolic challenges}

In this work, we address the problem of computing the Euler--Lagrange equations
for invariant Lagrangians using invariant calculus. This problem has previously 
been approached by many experts. Our approach is complementary in that it uses 
global invariants as well as develop practical techniques for computing 
the relative factor $W$. The latter allowed for the computation of 
the Euler--Lagrange equations for the conformal torsion functional, 
whose extremals (the conformal geodesics) are all singular.

A naive approach to compute the Euler--Lagrange equation in invariant terms
is to calculate it in jet-coordinates and then express the result through differential invariants. This approach has high computational complexity. 
For instance, one can effectively solve the general invariant problem 
for two-dimensional Euclidean motion group, 
but computations in three and especially in four dimensions are too difficult.

One can rely on the classical method of moving frames, using the explicit 
Lie group action, and \cite{kogan2003invariant} effectively addressed 
this problem for the three-dimensional 
Euclidean motion group. Yet in a similar problem in four dimensions, we encountered 
computational difficulties, meaning the symbolic system (we used the Differential Geometry package in Maple) does not terminate in a reasonable time. 
Even in dimension two, for the projective group the prolongation
of the action fails to terminate for 8-jets and higher.

In the latter case, the prolongation of the projective action of the 
corresponding Lie algebra can be computed directly and takes almost no time. 
The approach of replacing the Lie group action and moving frame with the 
Lie algebra action, together with a local cross-section, becomes more effective.
This was already noticed in the work of Kogan--Olver and their collaborators, 
where the recursive formulas for the invariant differentials do not depend 
on the Lie group action.

Let us note that globally invariant variational problems, in the case that $G$ 
acts locally freely and allows only a local transversal, can still be treated 
by the methods of moving frames as elaborated by Kogan--Olver. 
For this one may feed to their invariantization map the restrictions of global rational invariants to this cross-section, the resulting Euler--Lagrange
equation is globally invariant.

We take it further and apply it outside the moving frame technique. 
In fact, with our approach we do not rely on the way the generators of 
the algebra of differential invariants are computed. 

In our work we discussed 
how our method allows to overcome computational difficulties that arise 
in symbolic manipulations with Euler--Lagrange equations of high 
order via the ideas of symmetry. In particular, this allowed us to complete 
the computation of examples presented in this paper.

\end{document}